\documentclass[11pt]{amsart}

\usepackage[margin=2.4cm]{geometry}
\usepackage{amsmath}
\usepackage{mathabx}
\usepackage{amsfonts}
\usepackage{amsthm}
\usepackage{amssymb}
\usepackage{bm}
\usepackage{dsfont}
\usepackage[all]{xy}
\usepackage{eurosym} %simbolo euro
\usepackage{ulem}
\usepackage{pgfgantt}
\usepackage{color}
\usepackage{url}
\usepackage[T1]{fontenc}
\usepackage{stackrel}
\usepackage{float}

\usepackage{fancyhdr}

\newtheorem{theorem}{Theorem}
\newtheorem*{theorem*}{Theorem}
\newtheorem{definition}{Definition}
\newtheorem{lemma}{Lemma}
\newtheorem{prop}{Proposition}
\newtheorem{question}{Question}

\newtheorem{example}{Example}
\newtheorem{rk}{Remark}
\newtheorem{cor}{Corollary}
\newtheorem{nota}{Notations}

\newcommand{\N}{\mathds{N}}
\newcommand{\T}{\mathds{T}}
\newcommand{\Z}{\mathds{Z}} 
\newcommand{\R}{\mathds{R}}
\newcommand{\Q}{\mathds{Q}}

\title{Integrability  for conformally symplectic systems\\}
\author[M.-C. Arnaud] {Marie-Claude Arnaud}
\address{Universit\'e de Paris and Sorbonne Universit\'e, CNRS, IMJ-PRG, F-75006 Paris, France}
\email{marie-claude.arnaud@imj-prg.fr}
\thanks{$\ddag$   member of the {\sl Institut universitaire de France.}}

\author[X. Su] {Xifeng Su}
\address{School of Mathematical Sciences, Laboratory of Mathematics and Complex Systems (Ministry of Education)\\
Beijing Normal University\\
No. 19, XinJieKouWai St.,HaiDian District\\
 Beijing 100875, P. R. China}
\email{xfsu@bnu.edu.cn, billy3492@gmail.com}

\author[M. Zavidovique] {Maxime Zavidovique}
\address{Sorbonne Universit\'e, Universit\'e de Paris Cit\'e, CNRS, Institut de Math\'ematiques de Jussieu-Paris Rive Gauche, 75005 Paris, France}
\email{mzavidovi@imj-prg.fr}

\begin{document}

%\pagestyle{plain}
%\scriptsize
 
\normalsize

\maketitle

%{\color{blue} QUESTION OF ANNA: IS OUR RESULT TRUE FOR TEH TIEM DEPENDING CASE? I THINK YES...
%
%STATEMENTS TO WRITE:\begin{itemize}
%\item  IF EVERY POSITIVE ORBIT IS RELATIVELY COMPACT AND IF THERE IS NO CONJUGATE POINT ON THE NON-WANDERING SET (BETTER: ON THE CLOSURE OF THE UNION OF THE $\omega$-LIMIT SETS), THEN THE AS MASLOV INDEX VANISH AT EVERY POINT
%\item GIVE THE EXEMLE OF THE PENDULUM TO SHOW THAT SOMETIMES THE AS MASLOV INDEX IS NOT ZERO. REMARK: IN SOME WEAK SENSE, HOWEVER, THIS EXAMPLE IS INTEGRABLE (FOR CLAUDE, TEH ATTRACTOR IS A CURVE....)
%\item SPEAK ABOUT DIFFERENT NOTION OF ATTRACTORS IN THE INTRODUCTION, QUOTE BIRKHOFF ATTRACTOR?
%\item ARE OTHER NOTION OF INTEGRABILITY TO BE CONSIDERED? EG INVARIANT FOLIATION?
%\item {\color{red} Maxime : If no conjugate points and no fixed points in the attractor invariant foliation ?}
%\item ASSUME OUR HAMILTONIAN REMAINS HOPF INTEGRABLE WHEN CHANGING THE COHOMOLOGY CLASS (THE FACTOR?), CAN WE SAY SOMTEHING (AS SYMPLECTIC INTEGRABILITY)?
%\item NOTE SOMEWHERE THAT THE DYNAMICS ON THE INVARIANT TORUS CAN BE A LOT OF THINGS (MANE EXAMPLE, WITH STRONG OUTSIDE CONTARCTION), CONRARILY TO THE CASE OF SYMPELCTIC INTEGRABILITY)
%\end{itemize}
%}
%
%\vspace{15pt}

\section*{Introduction}

The goal of this paper is to study the dynamics of conformally symplectic Hamiltonian flows under the light of integrability. As the dynamics of conformally symplectic Hamiltonian flows are dissipative and differ fundamentally  from their conservative counterpart we start by proposing several notions of integrability that are better suited to the problem. We will propose two notions of integrability: 
$C^1$-integrability and Hopf integrability, that depend on the existence of a global attractor and on its shape. 

Then our main theorem focuses on Tonelli Hamiltonians whose conformally symplectic flows do not have conjugate points. We prove that such flows are automatically Hopf integrable.  The proof is geometric and studies the long time evolution of vertical subspaces under the flow. It also makes use of (discounted) weak KAM theory. We also establish several results about the asymptotic Maslov index for integrable conformally symplectic Hamiltonian flows. Finally, we describe some examples to illustrate differences between symplectic and conformally symplectic Hamiltonian flows and to illustrate the pertinence of our definitions of integrability.

\subsection*{History and origins of the problem} The starting point of this story is probably the Hopf conjecture, proved by Burago and Ivanov \cite{BuragoIvanov1994}:

\begin{theorem*}
A Riemannian metric on the $n$ dimensional torus whose geodesic flow does not have conjugate points is flat (meaning that it is conjugated to a constant metric).
\end{theorem*}

In particular, the phase space $T^*\T^n$ is partitioned by invariant  $C^0$ Lagrangian tori. This later property is called $C^0$-integrability. It is weaker than integrability and was established by Heber \cite{heber}, for Riemannian metrics without conjugate points around the same time the Hopf conjecture was solved. 

The problem of generalizing this to a wider class of Hamiltonians was addressed in \cite{AABZ2015} where it is proved that Tonelli Hamiltonians without conjugate points on the torus are $C^0$-integrable. One missing feature of $C^0$-integrability, compared to integrability, is that it is not guaranteed that the dynamics restricted on each invariant torus is conjugated to a rotation. It is still open whether  Tonelli Hamiltonians without conjugate points on the torus are integrable. 
For symplectic twist maps  a similar theorem was proved in \cite{ChengSun1996,FlorioLeCa2021} in dimension 2 and in \cite{arco} in any dimension.  In dimension $2$, on the annulus a deeper study of $C^0$-integrable twist maps was carried out in \cite{Ar-Za1,Ar-Za2} (see also \cite{ZaBook}).

Hamiltonian flows that are conservative have a natural cousin: conformally symplectic Hamiltonian flows. Those do not preserve the ambient symplectic form, but have the effect of contracting it in positive time. In small dimension, the study of non conservative twist maps of the 2 dimensional annulus  dates back to Birkhoff who introduced the famous Birkhoff attractor (see \cite{LeCalvez1988}). Recently there has been an increasing interest in higher dimensional  versions, starting from \cite{MaroSorrentino2017,ArnFej2021}. The notion of Birkhoff attractor was extended to higher dimensional systems in \cite{AHV}.  
From a PDE point of view, such conformally symplectic Hamiltonian flows are closely linked to the so called discounted Hamilton-Jacobi equation (see \cite{LionsPV1987,DaviniFathiIturriagaZavidovique}) that will also play a key role in this paper.

The following paper therefore addresses the natural question of studying versions of integrability for conformally symplectic Hamiltonian systems {that are defined of the cotangent bundle of a closed manifold}. As such systems  can have at most one compact invariant Lagrangian submanifold  Hamiltonianly isotopic to  a graph, see \cite{ArnFej2021}, we will propose different definitions of integrability more suited to the problem.

\subsection*{Setting and main results}

We assume that $M$ is a $d$-dimensional closed manifold and that $\pi:T^*M\to M$ is its cotangent bundle that is endowed with its tautological 1-form $\lambda$ and its symplectic form $\omega=-d\lambda$.  If $H:T^*M\to \R$ is a $C^{1, 1}$ function (often called Hamiltonian) and $\alpha>0$, the $(\alpha, H)$-conformal Hamiltonian vector field $X_H^\alpha$ is defined by 
\begin{equation}\label{E1} \iota_{X_H^\alpha}\omega=\alpha\lambda +dH.\end{equation}
In other words, $X_H^\alpha$ is the sum of $\alpha$ times the Liouville vector field $Z_\lambda$ (that is defined by $\iota_{Z_\lambda}\omega=\lambda$) and the usual symplectic Hamiltonian vector field associated to $H$. This vector field is not always complete. We will prove in Proposition \ref{PropExistGlobAttr} that it is complete in positive times when $H$ is convex in the fiber direction and coercive.

When defined, the flow $(\varphi_t)$ of $X_H^\alpha$ is conformally symplectic, i.e. $\varphi_t^*\omega=e^{-\alpha t}\omega$ (see e.g. \cite{ArnFej2021}).

When $H$ is convex in the fiber direction and coercive, $X_H^\alpha$ has always a (compact) {\sl global attractor}, see Definition \ref{defglobattr} and Proposition \ref{PropExistGlobAttr}. Because the condition of having a global attractor is invariant by conjugacy, there are other cases  where there is a global attractor. The simplest example is the  Liouville vector field for which the zero section is the global attractor. There also exist conformal Hamiltonian vector fields that have no global attractor, see Example \ref{ExNoglobattr}.

Our first goal is to provide appropriate notions of integrability in the dissipative conformally symplectic case on $T^*M$ in the presence of a global attractor while in the conservative Hamiltonian case, integrability means that the cotangent bundle is foliated by invariant Lagrangian submanifolds.

\begin{definition}[Integrability]\rm
The conformally Hamiltonian vector field $X$  is \begin{itemize}
\item {\sl $C^1$-integrable} if there is a   global attractor $\mathcal{A} \subset T^*M$ that is a symplectically isotopic to the zero section $C^1$   submanifold;
\item Hopf integrable if  there is a  global attractor and this global attractor is a Lipschitz Lagrangian  graph.
\end{itemize}
\end{definition}

We will give an example of a $C^1$-integrable conformal Hamiltonian vector field  that is not Hopf integrable, see Section \ref{C1notHopf}. We also provide an example of a Hopf integrable conformal Hamiltonian vector field  that is not $C^1$-integrable, see Section \ref{HopfnotC1}. {Also, observe that there exist Tonelli Hamiltonians that are not Hopf integrable nor $C^1$-integrable: consider the case of the damped pendulum, see e.g. \cite{MaroSorrentino2017}. Can the attractor be more complicated than in this case? For twist maps there exist examples with a complicated Birkhoff attractor (it is an indecomposable continuum), but their suspension is not Tonelli and we have no example of such a complicated attractor for a Tonelli Hamiltonian.}
%\begin{question}
%Does there exist an example of conformal Hamiltonian vector field that is Hopf integrable but not $C^1$ integrable? And if moreover we assume some restriction on $H$ (e.g. convex in the fiber direction)?
%\end{question}

In 1994 and in the symplectic Hamiltonian setting, \cite{BuragoIvanov1994}, Burago and Ivanov proved Hopf's conjecture: a Riemannian metric on $\T^d$ that has no conjugate point (see Definition \ref{Defconjvect} ) is flat and so the dynamics is completely integrable.  For Tonelli Hamiltonians (see Definition \ref{DefTonelli}), this result was partially extended in 2015 in \cite{AABZ2015}.

Our first result gives a sufficient condition for a Tonelli conformal Hamiltonian vector field  to be Hopf integrable (conjugate points are defined in Definition \ref{Defconjvect} in Subsection \ref{ssecdefHopf}).

\begin{theorem}\label{ThHopf}
Let $H:T^*\mathbb{T}^d\to \R$ be a Tonelli Hamiltonian and let $\alpha>0$. 
Assume that the flow  of $X_H^\alpha$ has no conjugate points. Then, $X_H^\alpha$ is Hopf integrable.
\end{theorem}

We will give an example of a Hopf and $C^1$ integrable $X_H^\alpha$ such that $H$ is Tonelli and there are conjugate points, see Section \ref{ssecExHopfintConjPts}. 

\begin{question} Is the conclusion of Theorem \ref{ThHopf} true if
\begin{itemize}
\item the global hypothesis is replaced by a local one: we assume that there is a neighborhood of the global attractor that contains no pair of conjugate points;
\item we don't assume that $H$ is Tonelli (maybe with some mild hypotheses, as convexity and coercivity)?
\end{itemize}
\end{question}

\begin{question}
Is it possible to be Hopf integrable with pairs of conjugate points in every neighborhood of the global attractor? 
\end{question}

We then raise the question of a possible reciprocal statement. The example of Section \ref{ssecExHopfintConjPts} implies that the conclusion of such a statement cannot be that there is no conjugate point at all. The pertinent notion for such a result is the one of {\sl dynamical Maslov index}. Useful references on these topics are the book of Maslov and specially the Appendix of Arnol'd, \cite{Arnold1972} and the articles \cite{CGIP2003, ArnaudFlorioRoos2022}. Roughly speaking, the dynamical Maslov index says how many times in (temporal) mean the images by the linearized flow of some Lagrangian subspaces  cross the tangent spaces to the (vertical) fibers. For example, an orbit with no conjugate points has zero dynamical Maslov index, but there exist orbits with zero dynamical Maslov index that have conjugate points. Our result is 
\begin{theorem}\label{ThintzeroasymptMasInd}
Let $H:T^*M\to \R$ be a   Hamiltonian and let $\alpha>0$. Assume that $X^\alpha_H$   is Hopf integrable or $C^1$-integrable. Then the dynamical Maslov index of every orbit is $0$. 
\end{theorem}
 Another interesting question is about the relations between the notion of integrability for the symplectic Hamiltonian vector field $X_H$ and the notion of integrability of the conformal Hamiltonian vector field $X_H^\alpha$. Does the integrability of one of the two implies the integrability of the other?\\
 We recall that $X_H$ is said to be integrable if there is a foliation of $T^*M$ into graphs whose  leaves are flow invariant Lagrangian submanifolds. \\
 The example of Section \ref{ssecExHopfintConjPts} is an example where $X_H$ is not symplectically integrable but $X_H^\alpha$ is integrable for all small enough $\alpha>0$.

%{\color{blue} DO WE HAVE AN EXAMPLE WHERE $X_H$ IS INTEGRABLE BUT NOT $X_H^\alpha$?}

\subsection{Notations}\begin{itemize}
\item We will use the notation ${\mathcal V}(q)=T^*_qM$ and $V(x)=T_x{\mathcal V}(q)\subset TT^*M$ for $x=(q,p)\in T^*M$. 
\item If $(\varphi_t)$ is a flow, if $x\in T^*M$, we introduce the notation $G_t(x)=D\varphi_t\big(\varphi_{-t}(x)\big)\Big(V\big(\varphi_{-t}(x)\big)\Big)$ for $t\in\R$.  
\item When $H:T^*M\to\R$ is Tonelli (see below), a Lagrangian function $L:TM\to \R$ is associated to $H$. It is defined by $L(q,v) = \max_{p\in T^*_q M}p(v) - H(q,p)$. The function $\mathcal L_H : TM \to T^*M$ given by $(q,v) \mapsto \left(q, \partial_v L(q,v)\right)$ is called the Legendre  map and is a $C^1$-diffeomorphism whose inverse is $\mathcal L^{-1}_H : T^*M \to TM$ given by $(q,p) \mapsto \left(x, \partial_p H(q,p)\right)$.
The conformal Euler-Lagrange flow with factor $\alpha>0$ is the flow generated by the vector field $ Y^\alpha_L = \mathcal L_{H*} X^\alpha_H$.
Finally, the equality $H\left(q,\partial_v L(q,v)\right)+L(q,v) = \partial_v L(q,v)(v)$ holds for all $(q,v)\in TM$.
 
\end{itemize}

\subsection{Several definitions}\label{ssecdefHopf}
Those definitions will be regularly used throughout the paper.
\begin{definition}\rm
A Hamiltonian $H$ is said to be
\begin{itemize}
\item  convex in the fiber direction if for every $q\in M$, the restriction of $H$ to the linear space $T^*_qM$ is convex;
\item  coercive if for every $C\in \R$, there exists a compact subset $K\subset T^*M$ such that
$$\forall (q, p)\in T^*M\backslash K, \quad H(q, p)>C.$$
\end{itemize}  
\end{definition}

\begin{definition}\label{DefTonelli}\rm
A $C^2$ Hamiltonian $H:T^*M\to \R$ is {\sl Tonelli} if
\begin{itemize}
\item $H$ is $C^2$ convex in the fiber direction, i.e. at every point the Hessian $\partial^2_{p p}H(q, p)$ is positive definite;
\item $H$ is superlinear in the fiber direction, i.e.
$$\lim_{\|p\|\to\infty}\frac{H(q, p)}{\| p\|}=+\infty.$$
\end{itemize}
\end{definition}

\begin{definition}\label{Defconjvect}\rm
Let   $X_H^\alpha$ be a conformal Hamiltonian vector field and let $(\varphi_t)$ be its flow. Two non-zero vectors $u, v\in T(T^*M)$ are {\sl conjugate} if
\begin{itemize}
\item $D\pi (u)=0$, \ $D\pi(v)=0$;
\item $\exists T\neq 0,\quad D\varphi_T u=v$.
\end{itemize}
Then we say also that the pair $(x,y)\in (T^*M)^2$ such that $u\in T_x(T^*M)$, $v\in T_y(T^*M)$ is a pair of  {\sl conjugate points}. In other words, $(x,y)$ is a pair of conjugate points  if for some $T\neq 0$, we have $D\varphi_T(x){  V}(x)\cap {  V}(y)\neq \{0\}$.

\end{definition}

\subsection{Reminders of weak KAM theory}We consider the following discounted Hamilton-Jacobi equation:
\begin{equation}\label{discountedHJequation}
\alpha u(q) + H\big(q, D u(q) \big) =0
\end{equation}where $\alpha>0$. 
When $H$ is coercive, there is a unique viscosity  solution $u_\alpha$ of \eqref{discountedHJequation}. When $H$ is Tonelli, $u_\alpha$ is the discounted weak KAM solution and  can be represented by the following formula
\begin{equation}\label{formuladiscweakKAMsolution}
u_\alpha(q)  = \inf_{\gamma} \int_{-\infty}^0 e^{\alpha s}  L\big(\gamma(s), \dot{\gamma}(s)\big)  \ ds, \qquad \forall ~q\in M
\end{equation}
where the infimum is taken over all absolutely continuous curves $\gamma: (-\infty, 0] \rightarrow M$ with $\gamma(0) =q$ and $L$ is the Lagrangian function associated to $H$.
 \begin{rk}\label{rkexactLagweakKAM}\rm Observe that if $u:M\to\R$ is a $C^{1, 1}$ function, i.e. a differentiable function whose derivative is Lipschitz, and if the graph of $Du$ is invariant by the  flow of $X_H^\alpha$ for some $H:T^*M\to\R$ that is Tonelli and $\alpha>0$, then $u$ satisfies a discounted Hamilton-Jacobi  equation 
\begin{equation}\label{discountedHJequationconstant}
\alpha u(q) + H\big(q, Du(q)  \big) =C.
\end{equation} 
Hence, there exists a constant $\kappa\in \R$ such that $u+\kappa$ satisfies \eqref{discountedHJequation} and is the discounted weak KAM solution.
Indeed, at Lebesgue almost every $q\in M$, $Du$ is differentiable. Because $X_H^\alpha$ is tangent to the graph of $Du$ and this graph is $C^0$-Lagrangian, we have at Lebesgue almost every $q\in M$
$$\forall v\in T_qM, \quad \omega\big(X_H^\alpha, (v, D^2u(q)v)\big)=0,$$
hence because $i_{X_H^\alpha}\omega=DH+\alpha\lambda$, 
 $$\forall v\in T_qM, \quad DH\big(q, Du(q)\big)(v, D^2u(q)v)+\alpha D u(q)v=0.$$
 \end{rk}

\subsection{Structure of the article}
In Section \ref{Sattractors}, we will explain the concept of  global attractor and give conditions that imply the existence of such a global attractor. In Section \ref{ssecExHopfintConjPts}, we will give an example of a Tonelli Hamiltonian that is Hopf and $C^1$ integrable but has conjugate points. In section \ref{SProoThHopf}, we will prove Theorem
 \ref{ThHopf}. In Section \ref{SProofTH2}, we will prove Theorem \ref{ThintzeroasymptMasInd}. In section \ref{C1notHopf}, we provide a
 $C^1$ integrable example  that is not Hopf integrable (and so has conjugate points). In section \ref{HopfnotC1}, we give an example of a Tonelli Hamiltonian that is Hopf integrable but not $C^1$ integrable.

%{\color{green} (OLD VERSION)We plan to look for a relevant notion of integrability in the dissipative conformally case on a cotangent bundle $\mathcal M=T^*M$ and to study connection between integrability and conjugate points (or dynamical Maslov index). For instance,
%
%\begin{conjecture}[Hopf]
%No conjugate points $\Longrightarrow$ integrability ?
%\end{conjecture}
%
%\begin{conjecture}
%Asymptotic Maslov index $= 0$ $\Longleftrightarrow$ integrability ?
%\end{conjecture}}

\section{Attractors}\label{Sattractors}
%{\color{blue} FAUT-IL METTRE LE FAIT QUE PRESQUE TOUT POINT EST DANS L'ENSEMBLE INSTABLE DE L'INFINI?}{\color{red} je veux bien, je sais pas faire ! (Maxime)}
%

The following Proposition shows why it is only relevant to look for compact attractors in positive time.

Let $H :T^*M \to \R$ be a $C^{1,1}$ Hamiltonian and $\alpha >0$. Let $\varphi : \mathcal D_H \to T^*M$ be its $\alpha$-conformally symplectic flow where $  \mathcal D_H\subset \R \times T^*M$ is the domain of the flow (that is not necessarily complete). If $x\in T^*M$, we set $I_x\subset \R$ the maximal interval on which the orbit of $x$ is defined.

\begin{prop}
 Then there is a set $A\subset T^*M$ of full Lebesgue measure such that all orbits starting in $A$ exit all compact sets in negative time:  for all $x\in A$, for all $K\subset T^*M$ compact subset, there is $0> t_K \in I_x$ such that  $\varphi_t(x)\notin K$ for all $t\in I_x$ such that $t<t_K$.
%{\color{red} MC: I WOULD PREFER: all $t\in I_x$ SUCH THAT $t<t_K$
% } 
 \end{prop}

\begin{proof}
We argue by contradiction. Assume there is a positive measure set $B \subset T^*M$ and a compact set $K\subset T^*M$ such that for all $x\in B$, the flow $\varphi_t(x)$ is defined for all $t<0$ and there is a sequence $t_n\to -\infty $ such that $\varphi_{t_n}(x) \in K$ for all $n$. 
%By regularity of the Lebesgue measure there exists $K_0 \subset B$ compact with $Leb(K_0)>0$. 
%And up to enlarging $K$ we assume $K_0 \subset K$. 
 By compactness there exists $\delta>0$ such that the flow is uniformly defined for $t\in [-\delta, \delta]$, for all $x\in K$. We set $K_\delta = \bigcup_{t\in [-\delta, \delta]} \varphi_t(K)$ that is also compact. Note that if $x\in B$, there is an increasing sequence of integers $ k_n \to +\infty$ such that $\varphi_{-k_n\delta}(x) \in K_\delta$.

Let now $\widetilde B\subset T^*M$ the set of $x\in T^*M$ such that $ (-\infty , 0] \subset I_x $ and there is an increasing sequence of integers $ k_n \to +\infty$ such that $\varphi_{-k_n\delta}(x) \in K_\delta$. This set clearly verifies $\varphi_{-\delta}(\widetilde B) \subset \widetilde B$. Moreover it has finite Lebesgue measure. Indeed $\widetilde B \subset \bigcup_{k\geqslant 0} \varphi_{k\delta} (K_\delta^k)$ where $K_\delta^k$ is the set of $x\in  K_\delta$ such that $ (-\infty, k\delta] \subset I_x $. As the flow is conformally symplectic, $Leb\big( \bigcup_{k\in \N} \varphi_{k\delta} (K_\delta^k)\big) \leqslant \frac{1}{1-e^{-d\alpha\delta }}Leb(K_\delta)$. It follows that $Leb(\widetilde B)\geqslant Leb\big(\varphi_{-\delta}(\widetilde B)\big) = e^{d\alpha \delta} Leb(\widetilde B)$ leading to $Leb(\widetilde B)=0$, a contradiction as $B\subset \widetilde B$.

%
%We now partition $K_0$ according to the first entry time in $K_\delta$ setting for $k>0$, 
%$$A_k = \{ x\in K_0 , \ \ \varphi_{-k\delta}(x) \in K_0 \ \ {\rm and}\ \   \forall 1\leqslant j \leqslant k-1, \  \varphi_{-j\delta}(x) \notin K_0  \}.$$ 
%Note that the sets $\varphi_{-k\delta}(A_k)$ are also pairwise disjoint. Indeed, if by contradiction  $x\in \varphi_{-k\delta}(A_k)\cap \varphi_{-k'\delta}(A_{k'})$ for some $1\leqslant k<k'$. Then $\varphi_{-(k'-k)\delta} \big(\varphi_{k'\delta}(x)\big) = \varphi_{k\delta}(x) \in A_k \subset K$ thus contradicting the minimality of $k'$ as  $\varphi_{k'\delta}(x)\in A_{k'}$.
%
%
%Ensemble des points qui reviennent une infinité de fois dans $K_\delta$ de mesure finie (inclus dans union des images itérées) et invariant donc de mesure nulle.
\end{proof}

\begin{definition}\label{defglobattr}\rm
Let $(\varphi_t)$ be a flow defined for $t\geq 0$ on a manifold $N$ and let $K$ be a non-empty compact subset of $N$. 
\begin{itemize}
\item $K$ is an attractor if there is an open set $U$ such that $K\subset U$, for every $t>0$, $\varphi_t(\overline{U})\subset U$ and $K=\bigcap\limits_{t>0}\varphi_t(U)$.
\item $K$ is the global attractor if it is an attractor and if for every neighborhood $V$ of $K$ and every $x\in N$, there exists $t\in\R$ such that $\varphi_t(x)\in V$.
\end{itemize}
\end{definition}

\begin{rk}\rm
For the first point, the condition that $\varphi_t(\overline{U})\subset U$ can be removed. Indeed, assume only the existence of  an open neighborhood $U$ of $K$ such that  $K=\bigcap\limits_{t>0}\varphi_t(U)$. Let $V\subset N$ be an open set such that $K\subset V \subset \overline V \subset U$. Then $K=\bigcap\limits_{t>0}\varphi_t(\overline V)$.  It is proved in \cite[Theorem 1.4]{Fathi22} that there exists a smooth Lyapunov  function $f : N \to [0,+\infty)$ such that $f^{-1}\{0\} = K$ and $Df(x)\cdot X(x) <0$ for all $x \in \overline V\setminus K$ where $X$ denotes the vector field generating the flow. Then setting $U_\varepsilon = f^{-1}[0,\varepsilon)$ for $\varepsilon>0$ small enough yields an open set that contains $K$ and such that  $\varphi_t(\overline{U_\varepsilon})\subset U_\varepsilon$ and $K=\bigcap\limits_{t>0}\varphi_t(U_\varepsilon)$. 
\end{rk}

\begin{rk}\rm
For the second point of the definition, we just need to check that the condition is satisfied for $V=U$ for a $U$ as in the first point. 
\end{rk}

\begin{prop}\label{Propuniglobattr}
There is always at most one global attractor. 
\end{prop}
\begin{proof}[Proof of Proposition \ref{Propuniglobattr}]
We assume that $K$ is a global attractor for the flow $(\varphi_t)$  and we choose $U$ as in Definition \ref{defglobattr}. 

From $K=\bigcap\limits_{t>0}\varphi_t(\overline{U} )$ and the fact that $t\mapsto \varphi_t(\overline U)$ is decreasing, we deduce that for every neighborhood $V$ of $K$, there exists $t>0$ such that $\varphi_t(\overline{U})\subset V$. Hence, replacing $U$ by $\varphi_t(U)$, we can assume that $U$ is as close as we want to $K$.

We now assume that $K_1\neq K_2$ are two global attractors. We can assume for example that $K_2\backslash K_1\neq \emptyset$. Then we can choose an open subset $U_1$ of $N$ such that $K_2\not\subset  U_1$,  for every $t>0$, $\varphi_t(\overline{U_1})\subset U_1$ and $K_1=\bigcap\limits_{t>0}\varphi_t(U_1)$.\\
 Because $K_1$ is a global attractor and $U_1$ is forward invariant, for every $x\in K_2$, there exists $t_x>0$ such that $\forall t\geq t_x, \  \varphi_t(x)\in U_1$. Hence  $\displaystyle{ K_2\subset \bigcup _{t>0}\varphi_{-t}(U_1)}$. As we have an increasing family of open sets, there is a $T>0$ such that $ K_2\subset \varphi_{-T}(U_1)$ and then
 $\varphi_T(K_2)\subset U_1$. This is a contradiction with the invariance of $K_2$ and the fact that $K_2\nsubset  U_1$.
\end{proof}

\begin{cor} When $N$ is compact, the unique global attractor is $N$. 

\end{cor} 
\begin{prop}\label{PropExistGlobAttr}
Assume that $H:T^*M\to \R$ is $C^{1, 1}$,  convex in the fiber direction and coercive. Then  the $(\alpha, H)$ conformal Hamiltonian vector field $X_H^\alpha$ is complete in positive times and has a global attractor.
\end{prop}

A different proof of this proposition is given in \cite{MaroSorrentino2017} when $H$ is assumed to be Tonelli, that is a more restrictive condition than ours.

\begin{proof}[Proof of Proposition \ref{PropExistGlobAttr}]
 Let $H:T^*M\to\R$ be $C^{1, 1}$,  convex in the fiber direction and coercive. We prove that there exists $R>0$ such that $H$ is a strict Lyapunov function on $\{ H>R\}$. We can then conclude that $X_H^\alpha$ is complete in positive times and that $\bigcap\limits_{t>0}\varphi_t(\{H\leq R+1\})$ is the global attractor (where $(\varphi_t)$ is the flow of $X_H^\alpha$).
 
 We use canonical coordinates $(q, p)$. 
 \begin{lemma}\label{LemmLyapfunct} With the hypothesis of Proposition \ref{PropExistGlobAttr}, assume that $H(q, p)>R=\sup_{q\in M} H(q, 0)$ and denote by $\big(q(t), p(t)\big)$  the orbit of $(q, p)$. Then $\frac{d}{dt}\big(H(q(t), p(t)\big)_{|t=0}<0$.
 \end{lemma}
 \begin{proof}[Proof of Lemma  \ref{LemmLyapfunct}]
 Because of  convexity, we have 
  $$\frac{d}{dt}\big(H(q(t), p(t)\big)_{|t=0}=-\alpha \partial_pH(q, p)p\leq\alpha\big(H(q,0)-H(q,p)\big)< 0. \qedhere$$
 \end{proof}
 \end{proof}
We deduce from the last proof the following statement.
\begin{cor}\label{Corcritsimpleint}
Assume that $H:T^*M\to \R$ is $C^{1, 1}$,   convex in the fiber direction, coercive and that the minimum of $H$ is attained exactly at all the points of the zero section.   Then  the $(\alpha, H)$ conformal Hamiltonian vector field $X_H^\alpha$ is Hopf and $C^1$ integrable and the zero section is the global attractor. \end{cor}
\begin{proof}
In  Lemma  \ref{LemmLyapfunct}, we proved that $H$ is a strict Lyapunov  function outside the zero section for $X_H^\alpha$ for every $\alpha>0$.  As  $X_H^\alpha$ vanishes on the zero section, the latter is a smooth invariant Lagrangian manifold and the global attractor and $X_H^\alpha$ is both Hopf and $C^1$ integrable.
\end{proof}

It can happen that there  is no global attractor.
\begin{example}\label{ExNoglobattr} Assume $\alpha\in (0, 2\pi)$ and define on $T^*\T=\T\times \R$, $H(q, p)=p\sin(2\pi q)$. Then we have
$$
\frac{dq}{dt}= \sin (2\pi q)\quad{\rm and}\quad\frac{dp}{dt}=-(2\pi\cos( 2\pi q)+\alpha)\ p.
 $$
 Then if $q(0)=\frac{1}{2}$, we have $ q(t)=\frac{1}{2}$ for all $t\in \R$ and $\frac{dp}{dt}=(2\pi-\alpha)\ p$ where $2\pi-\alpha>0$. Hence $t\in\R\mapsto e^{(2\pi-\alpha)t}$ is an orbit that goes outside every compact set when $t\to+\infty$.  
 Hence there is no global attractor.
 
% {\color{blue} IF WE SPEAK OF BIRKHOFF ATTRACTOR BEFORE, WE HAVE TO REMARK   HERE THAT IN THIS EXAMPLE THE  BIRKHOFF ATTRACTOR IS THE ZERO SECTION.}

\end{example}
We remark that  the Birkhoff attractor in this example is the zero section.

\begin{figure}
\begin{center}
\tikzset{every picture/.style={line width=0.75pt}} %set default line width to 0.75pt        
%\centering
\begin{tikzpicture}[x=0.75pt,y=0.75pt,yscale=-.8,xscale=1]
%uncomment if require: \path (0,326); %set diagram left start at 0, and has height of 326

%Straight Lines [id:da951253202738662] 
\draw    (148.5,161.5) -- (317.5,161.5) ;
\draw [shift={(237.2,161.5)}, rotate = 180] [color={rgb, 255:red, 0; green, 0; blue, 0 }  ][line width=0.75]    (7.65,-2.3) .. controls (4.86,-0.97) and (2.31,-0.21) .. (0,0) .. controls (2.31,0.21) and (4.86,0.98) .. (7.65,2.3)   ;
%Straight Lines [id:da0818975873209038] 
\draw    (317.5,161.5) -- (486.5,161.5) ;
\draw [shift={(396.8,161.5)}, rotate = 0] [color={rgb, 255:red, 0; green, 0; blue, 0 }  ][line width=0.75]    (7.65,-2.3) .. controls (4.86,-0.97) and (2.31,-0.21) .. (0,0) .. controls (2.31,0.21) and (4.86,0.98) .. (7.65,2.3)   ;
%Straight Lines [id:da4390075219110081] 
\draw    (148.5,32.5) -- (148.5,161.5) ;
\draw [shift={(148.5,101.2)}, rotate = 270] [color={rgb, 255:red, 0; green, 0; blue, 0 }  ][line width=0.75]    (7.65,-2.3) .. controls (4.86,-0.97) and (2.31,-0.21) .. (0,0) .. controls (2.31,0.21) and (4.86,0.98) .. (7.65,2.3)   ;
%Straight Lines [id:da4537237714768425] 
\draw    (486.5,30.5) -- (486.5,161.5) ;
\draw [shift={(486.5,100.2)}, rotate = 270] [color={rgb, 255:red, 0; green, 0; blue, 0 }  ][line width=0.75]    (7.65,-2.3) .. controls (4.86,-0.97) and (2.31,-0.21) .. (0,0) .. controls (2.31,0.21) and (4.86,0.98) .. (7.65,2.3)   ;
%Straight Lines [id:da571817319113501] 
\draw    (317.5,161.5) -- (317.5,292.5) ;
\draw [shift={(317.5,231.2)}, rotate = 270] [color={rgb, 255:red, 0; green, 0; blue, 0 }  ][line width=0.75]    (7.65,-2.3) .. controls (4.86,-0.97) and (2.31,-0.21) .. (0,0) .. controls (2.31,0.21) and (4.86,0.98) .. (7.65,2.3)   ;
%Straight Lines [id:da18731692190069738] 
\draw    (486.5,161.5) -- (486.5,288.5) ;
\draw [shift={(486.5,219.8)}, rotate = 90] [color={rgb, 255:red, 0; green, 0; blue, 0 }  ][line width=0.75]    (7.65,-2.3) .. controls (4.86,-0.97) and (2.31,-0.21) .. (0,0) .. controls (2.31,0.21) and (4.86,0.98) .. (7.65,2.3)   ;
%Straight Lines [id:da5152713773351558] 
\draw    (148.5,161.5) -- (148.5,293.5) ;
\draw [shift={(148.5,222.3)}, rotate = 90] [color={rgb, 255:red, 0; green, 0; blue, 0 }  ][line width=0.75]    (7.65,-2.3) .. controls (4.86,-0.97) and (2.31,-0.21) .. (0,0) .. controls (2.31,0.21) and (4.86,0.98) .. (7.65,2.3)   ;
%Straight Lines [id:da7796391833454395] 
\draw    (317.5,32.5) -- (317.5,161.5) ;
\draw [shift={(317.5,91.8)}, rotate = 90] [color={rgb, 255:red, 0; green, 0; blue, 0 }  ][line width=0.75]    (7.65,-2.3) .. controls (4.86,-0.97) and (2.31,-0.21) .. (0,0) .. controls (2.31,0.21) and (4.86,0.98) .. (7.65,2.3)   ;
%Curve Lines [id:da4638697904285798] 
\draw    (162.5,67.5) .. controls (164.5,118.5) and (265.5,196.5) .. (308.5,69.5) ;
\draw [shift={(239.11,138.35)}, rotate = 182.84] [color={rgb, 255:red, 0; green, 0; blue, 0 }  ][line width=0.75]    (7.65,-2.3) .. controls (4.86,-0.97) and (2.31,-0.21) .. (0,0) .. controls (2.31,0.21) and (4.86,0.98) .. (7.65,2.3)   ;
%Curve Lines [id:da6745117303125803] 
\draw    (181.5,56.5) .. controls (184.5,109.5) and (264.5,152.5) .. (290.5,58.5) ;
\draw [shift={(239.59,113.82)}, rotate = 181.36] [color={rgb, 255:red, 0; green, 0; blue, 0 }  ][line width=0.75]    (7.65,-2.3) .. controls (4.86,-0.97) and (2.31,-0.21) .. (0,0) .. controls (2.31,0.21) and (4.86,0.98) .. (7.65,2.3)   ;
%Curve Lines [id:da12743052511276565] 
\draw    (344.5,58.5) .. controls (350.5,113.5) and (432.5,156.5) .. (458.5,62.5) ;
\draw [shift={(395.47,116.71)}, rotate = 8.23] [color={rgb, 255:red, 0; green, 0; blue, 0 }  ][line width=0.75]    (7.65,-2.3) .. controls (4.86,-0.97) and (2.31,-0.21) .. (0,0) .. controls (2.31,0.21) and (4.86,0.98) .. (7.65,2.3)   ;
%Curve Lines [id:da1038290227024985] 
\draw    (330.5,80.5) .. controls (332.38,128.54) and (440.5,197.5) .. (473.5,82.5) ;
\draw [shift={(397.93,144.37)}, rotate = 7.25] [color={rgb, 255:red, 0; green, 0; blue, 0 }  ][line width=0.75]    (7.65,-2.3) .. controls (4.86,-0.97) and (2.31,-0.21) .. (0,0) .. controls (2.31,0.21) and (4.86,0.98) .. (7.65,2.3)   ;
%Curve Lines [id:da11225263206120761] 
\draw    (167.5,242.5) .. controls (214.5,95.5) and (308.5,226.5) .. (306.5,252.5) ;
\draw [shift={(242.83,176.63)}, rotate = 195.59] [color={rgb, 255:red, 0; green, 0; blue, 0 }  ][line width=0.75]    (7.65,-2.3) .. controls (4.86,-0.97) and (2.31,-0.21) .. (0,0) .. controls (2.31,0.21) and (4.86,0.98) .. (7.65,2.3)   ;
%Curve Lines [id:da9589866678748056] 
\draw    (174.5,263.5) .. controls (221.5,116.5) and (298.5,253.5) .. (296.5,279.5) ;
\draw [shift={(243.89,200.37)}, rotate = 201.83] [color={rgb, 255:red, 0; green, 0; blue, 0 }  ][line width=0.75]    (7.65,-2.3) .. controls (4.86,-0.97) and (2.31,-0.21) .. (0,0) .. controls (2.31,0.21) and (4.86,0.98) .. (7.65,2.3)   ;
%Curve Lines [id:da5467782163231103] 
\draw    (341.5,280.5) .. controls (388.5,133.5) and (465.5,270.5) .. (463.5,296.5) ;
\draw [shift={(402.09,214.27)}, rotate = 16.94] [color={rgb, 255:red, 0; green, 0; blue, 0 }  ][line width=0.75]    (7.65,-2.3) .. controls (4.86,-0.97) and (2.31,-0.21) .. (0,0) .. controls (2.31,0.21) and (4.86,0.98) .. (7.65,2.3)   ;
%Curve Lines [id:da18487355385924809] 
\draw    (335.5,248.5) .. controls (382.5,101.5) and (476.5,232.5) .. (474.5,258.5) ;
\draw [shift={(401.59,180.33)}, rotate = 11.73] [color={rgb, 255:red, 0; green, 0; blue, 0 }  ][line width=0.75]    (7.65,-2.3) .. controls (4.86,-0.97) and (2.31,-0.21) .. (0,0) .. controls (2.31,0.21) and (4.86,0.98) .. (7.65,2.3)   ;

% Text Node
\draw (123,167.4) node [anchor=north west][inner sep=0.75pt]  [font=\tiny]  {$( 0,0)$};
% Text Node
\draw (288,167.4) node [anchor=north west][inner sep=0.75pt]  [font=\tiny]  {$\left(\frac{1}{2} ,0\right)$};
% Text Node
\draw (490,172.4) node [anchor=north west][inner sep=0.75pt]  [font=\tiny]  {$( 1,0)$};

\end{tikzpicture}
\caption{A conformal Hamiltonian flow with no global attractor}
\end{center}
\end{figure}
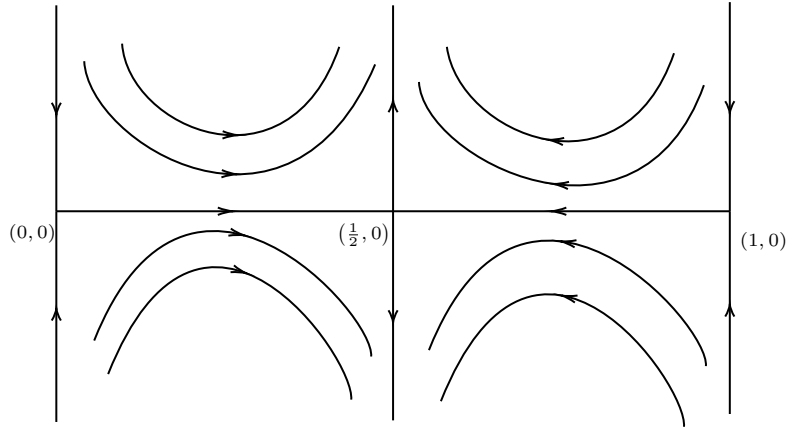

%\section{Hopf type theory}\label{SecHopf}
\section{Example of Tonelli Hamiltonian whose symplectic Hamiltonian vector field $X_H$ is not integrable but such that for all small enough $\alpha>0$, $X_H^\alpha$ has conjugate points and is    Hopf  and $C^1$ integrable }\label{ssecExHopfintConjPts}
%{\color{blue} I CHANGED THE EXAMPLE BECAUSE THERE WAS A PROBLEM: THE ZERO SECTION WAS NOT INVARIANT BY THE PREVIOUS EXAMPLE}

We work on $\T^d$. We assume that $\eta:\R^d\to [0, 1]$ is a smooth function that is $1$ on $B(0, r)$ and $0$ on $\R^d\backslash B(0, R)$. Choosing $R-r$ large, we can assume that $D\eta$ and $D^2\eta$ are small.\\
 Let $p_0=a_0/b_0$ be a point of  $\Q^d$ that doesn't belong to $B(0, R)$ where $a_0\in\Z^d$ and $b_0\in \N^*$. We build a small smooth non-positive potential $V:\T^d\to \R_+$ that vanishes only on the curve  $E=\{ tp_0 ; t\in [0, b_0]\}$ of $\T^d$ and whose Hessian at every point of  $E$ has rank $d-1$. This implies that this Hessian $\text{Hess} \, V(tp_0)$ is non-positive with a $1$-dimensional kernel. Identifying $T_{q}\T^d$ with $\R^d$, we will also assume that this Hessian is a constant symmetric matrix $\text{Hess}\, V(tp_0)=S\leq 0$ (that doesn't depend on $t$).
 
 We introduce the Hamiltonian $H(q, p)=\frac{1}{2}\| p\|^2+ \eta(p-p_0) V(q)$ of $T^*\T^d=\T^d\times \R^d$. As $D^2\eta$ is small and $\eta$ has compact support, $H$ is Tonelli. Hamilton's equations are

$$\begin{cases}
  \dot q=p+V(q)\nabla \eta(p-p_0);\\
  \dot p=-\eta(p-p_0)\nabla V(q).
\end{cases}$$
On $\widetilde{E}=\{ (tp_0, p_0); t\in [0, b_0]\}$, we have $V(tp_0)=0$, $\nabla V(tp_0)=0$,  $\eta(0)=1$ and $\nabla\eta (0)=0$. We deduce on $\widetilde E$ that 
$$\begin{cases}
  \dot q=p_0;\\
  \dot p=0
\end{cases}$$
and that $\widetilde E$ is a periodic orbit for the Hamiltonian flow of $H$. The Hessian of $H$ on $\widetilde E$ is 
$$\text{Hess}\,H(tp_0, p_0)= \begin{pmatrix}\eta(0) \text{Hess}\, V(tp_0)&\nabla V(tp_0)^{t}\nabla \eta(0)\\
\nabla\eta(0)^{t}\nabla V(tp_0)&{\mathbf 1}-V(tp_0)\text{Hess}\, \eta (0)\end{pmatrix}=\begin{pmatrix}S&{\mathbf 0}\\
{\mathbf 0}& {\mathbf 1}\end{pmatrix}.$$
The eigenvalues of 
$$J\text{Hess}\, H(tp_0, p_0)=\begin{pmatrix}{\mathbf 0}&{\mathbf 1}\\
-{\mathbf 1}& {\mathbf 0}\end{pmatrix}\begin{pmatrix}S&{\mathbf 0}\\
{\mathbf 0}& {\mathbf 1}\end{pmatrix}=\begin{pmatrix}{\mathbf 0}&{\mathbf 1}\\
-S& {\mathbf 0}\end{pmatrix}$$
are given by

\[\begin{split}J\text{Hess}\,H(tp_0, p_0)(u,v)=\lambda(u,v)&\Longleftrightarrow v=\lambda u \text{ and } Su=-\lambda v\\\
&\Longleftrightarrow v=\lambda u \text{ and } Su=-\lambda^2 u\\
&\Longleftrightarrow (S+\lambda^2)u=0\text{ and } v=\lambda u.
\end{split}\]
As $S$ is non-positive with rank $d-1$, $J\text{Hess}H(tp_0, p_0)$ has $(d-1)$ negative eigenvalues, $(d-1)$ positive eigenvalues and two zero eigenvalues. Hence  the periodic orbit $\widetilde{E}$ is hyperbolic, which implies that the Hamiltonian $H$ is not integrable and we deduce that he has conjugate points, see \cite{AABZ2015}.  Indeed, in \cite{AABZ2015}, it is proven that the periodic orbits of such Hamiltonians are contained in invariant Lagrangian graphs that are filled by periodic orbits with the same period. Hence, these periodic orbits have all their Floquet multipliers equal to 1 and cannot be hyperbolic. 

Moreover, as the Hamiltonian is Tonelli, the conjugate points remain by perturbation, because the orbit of the tangent to the vertical fiber at a certain  point is transverse to the singular Maslov cycle, see e.g. \cite{ArnaudFlorioRoos2022}. 

 Now, we consider the conformally Hamiltonian flow for a small $\alpha>0$.  It has then conjugate points.
 
% {\color{green} THIS IS IN FACT USELESS IF WE APPLY COROLLARY \ref{Corcritsimpleint}.  And we have
%$$\frac{d}{dt}H(q(t), p(t))=-\alpha \partial_pH(q, p)p=-\alpha(\| p\|^2 -V(q) d\eta(p-p_0)p).$$
%This vanishes only for $p=0$. Indeed,  $V(q)d\eta(p-p_0)p$ is non-zero only if $p\in B(p_0, R)\backslash B(p_0, r)$ but on this compact set, we can assume that $V$ is  small enough to  have  $|V(q) d\eta (p-p_0)p|<\| p\|^2$. Hence $H$ is a Lyapunov function that is strict outside the zero section. }

 As  $H(q,p)=H(q, p)=\frac{1}{2}\| p\|^2$ in a neighborhood of the zero section,  the minimum of $H$ on every fiber is attained on the zero section and by Corollary \ref{Corcritsimpleint},
 % the zero section is invariant and  
 the global attractor is the zero section.
\section{Proof of Theorem \ref{ThHopf}}\label{SProoThHopf}

The proof is performed lifting everything from the torus $\T^d$ to its universal cover $\R^d$ and  will be in three steps:
\begin{itemize}
\item In the first step we prove that all trajectories of the conformally symplectic flow satisfy a minimization property (Proposition \ref{PTonelli}). This makes use directly of the hypothesis of no conjugate points. From this we deduce that for all $q\in \R^d$,  $\varphi_t(\{q\}\times \R^d)$ is a smooth graph above $\R^d$ for all $t>0$, where $\varphi_t$ is the conformal flow (Corollary \ref{Corimagvert}).
\item Next we take $(q_0,p_0)$ in the global attractor, set $\varphi_s(q_0,p_0) = \big(q(s),p(s)\big)$ for $s\in \R$ and apply the previous result to $q = q(-t)$ for $t>0$. We prove that as $t\to +\infty$, $\varphi_t(\{q(-t)\}\times \R^d)$ converges to a $\Z^d$-periodic Lipschitz graph that is the differential of a $\Z^d$-periodic function $u_\alpha : \R^d\to \R^d$ solution to a Hamilton-Jacobi equation
(Proposition \ref{PFinale}).
\item Finally we prove that this graph of $D u_\alpha$ is the global attractor (Corollary \ref{CorHopf}).
\end{itemize}

 Let us prove Theorem \ref{ThHopf}.

We want to prove that no conjugate points implies Hopf integrability. We assume $L>0$ (we can do that by adding a constant) and 
assume that  there are no conjugate points on $T^*\mathbb{T}^d$. For $q_1, q_2\in \mathbb{R}^d$ and   $t>0$, we define
\begin{equation}\label{actionRn}
\widetilde{\mathcal{A}}_t(q_1, q_2) := \min_{\substack{\gamma: [-t, 0]\rightarrow \mathbb{R}^d \\ \gamma(-t) = q_1 ,\gamma(0) = q_2} } \int_{-t}^0 e^{\alpha s} L\big(\gamma(s), \dot{\gamma}(s)\big) \ ds
\end{equation}
$\widetilde{\mathcal{A}}_t$ is $\Delta_{\Z^d}$-periodic, i.e. \[
\widetilde{\mathcal{A}}_t(q_1+ k, q_2 +k) = \widetilde{\mathcal{A}}_t(q_1, q_2)  \text{  for all  }k\in \mathbb{Z}^d.
\]
 We cut the proof in different parts. 
 
 We recall that an extremal is a solution of the Euler-Lagrange Equation associated to $\mathfrak L(t,q, v) =e^{\alpha t}L(q,v)$.
 
\begin{prop}\label{PTonelli}
$\gamma\mapsto\int_{-t}^0 e^{\alpha s}L\big(\gamma(s), \dot\gamma(s)\big)ds$ has a unique extremal among arcs joining $q_1$ to $q_2$, it is  where the minimum is attained.
\end{prop} 
\begin{proof}[Idea of proof of Proposition \ref{PTonelli}] 
As the proof is very similar to the proof of Proposition 1 in \cite{AABZ2015} we will not give all the details. In what follows, we will write similarly functions on $\R \times \T^d \times \R^d$ or their lift to $\R \times \R^d \times \R^d$, the context will make it clear what is considered.
 
Notice that the minimum in the definition of  \eqref{actionRn} exists by classical Tonelli theory, applied to the (time dependent) Lagrangian function  
\begin{equation}\label{ELtimedepending}\mathfrak L(t,q, v) =e^{\alpha t}L(q,v).\end{equation} 
Curves realizing the minimum are Lipschitz continuous by Clarke-Vinter theory (see \cite{ClarkeV} Remark 5, page 1716) and then of class $C^2$ (\cite{Clarke} Theorem 5.7). They are hence extremal curves\footnote{that is a piece of trajectory of the Euler-Lagrange flow of $\mathcal L$.}.

Let us also recall that if $\gamma : [-t,0] \to \R^d$ is an extremal curve, then the following relations hold:
$$\forall s\in[-t,0], \quad \varphi_s\Big(\gamma(0),\partial_vL \big(\gamma(0),\dot\gamma(0)\big)\Big) = \Big(\gamma(s),\partial_vL\big(\gamma(s),\dot\gamma(s)\big)\Big).
$$
We need a standard a priori Lipschitz estimate on minimizers, however the proof needs to be adapted as the flow is not necessarily complete.

\begin{lemma}\label{aprioricompactness}
Given $t>0 $ and $C>0$,  there exists $R>0$ such that any extremal curve $\gamma : [-t, 0] \to \R^d$ with action $ \int_{-t}^0 e^{\alpha s} L\big(\gamma(s), \dot{\gamma}(s)\big) \ ds < C+1$ is automatically $R$-Lipschitz. Moreover, by taking $R$ sufficiently large, we may enforce that the same holds true for any Tonelli Lagrangian $\widetilde{ \mathfrak L}$ that coincides with $\mathfrak L$ on $[-R,R] \times \T^d \times B(0,R)$.
\end{lemma}

\begin{proof}
By the mean value theorem there exists $s_0\in [-t,0]$ such that $L\big(\gamma(s_0), \dot{\gamma}(s_0)\big)\leqslant \frac{e^{ \alpha t}(C+1)}{t}$. Hence by superlinearity of $L$, there exists $R_0>0$ only depending on $L$, on $C$ and on $t$ such that $\| \dot{\gamma}(s_0) \|< R_0$. 

Let us set $p(s) = \partial_v L\big(\gamma(s),\dot \gamma(s)\big)$ for $s\in [-t,0]$. The computation in Proposition  \ref{PropExistGlobAttr} shows that $\frac{d}{ds}H\big(\gamma(s),p(s)\big) = -\alpha\partial_p H\big(\gamma(s),p(s)\big)p(s) \leqslant \beta -\alpha H\big(\gamma(s),p(s)\big)$ where $\beta = \alpha\max_q H(q,0)$. It follows from Gronwall's lemma that for $s>s_0$, 
$H\big(\gamma(s),p(s)\big)\leqslant e^{-\alpha(s-s_0)}\Big(\frac{\beta}{\alpha}(e^{\alpha(s-s_0)}-1) + H\big(\gamma(s_0),p(s_0)\big)\Big)$.
% {\color{blue}[ I obtain that $H(\gamma(s),p(s))\leqslant e^{-\alpha(s-s_0)}(\frac{\beta}{\alpha}(e^{\alpha(s-s_0)}-1) + H(\gamma(s_0),p(s_0))$. As the right term satifies the equality, I HAVE A DOUBT ON YOUR INEQUALITY]} 
 Now let $R_1>0$ such that $H\big(q, \partial_vL(q,v)\big) <R_1$ for all $(q,v)\in \R^d\times B(0,R_0)$. Then we obtain that $H\big(\gamma(s),p(s)\big)\leqslant \frac{\beta}{ \alpha}+R_1$ for $s>s_0$.
% {\color{blue} INSTEAD $H(\gamma(s),p(s))\leqslant \frac{\beta}{\alpha} +R_1$?}

For $s<s_0$ we set $f(s) =  H\big(\gamma(s),p(s)\big)$  and we discover that 
\begin{multline*}f'(s) = -\alpha\partial_p H\big(\gamma(s),p(s)\big)p(s) \\
= -\alpha \Big(H\big(\gamma(s),p(s)\big)+L\big(\gamma(s),\dot \gamma(s)\big)\Big) =\\
 -\alpha f(s) - \alpha L\big(\gamma(s),\dot \gamma(s)\big).
\end{multline*}
It follows that $f(s)=e^{-\alpha(s-s_0)}f(s_0) -\alpha \int_{s_0}^s e^{\alpha(\sigma-s)}L\big(\gamma(\sigma),\dot \gamma(\sigma)\big)d\sigma$ from which we get
% {\color{blue} I FIND $f(s)=e^{-\alpha(s-s_0)}f(s_0) -\alpha \int_{s_0}^s e^{\alpha(\sigma-s)}L(\gamma(\sigma),\dot \gamma(\sigma))d\sigma$. }

\begin{multline*}
H\big(\gamma(s),p(s)\big) =  e^{-\alpha(s-s_0)}H\big(\gamma(s_0),p(s_0)\big)+\alpha e^{-\alpha s} \int_{s}^{s_0} e^{\alpha\sigma}L\big(\gamma(\sigma),\dot \gamma(\sigma)\big)d\sigma\\
\leqslant  e^{-\alpha(s-s_0)} H\big(\gamma(s_0),p(s_0)\big)+\alpha (C+1)<e^{\alpha t} \big(R_1+\alpha (C+1)\big).
\end{multline*}
We used that as $L\geqslant 0$, then  $\int_{s}^{s_0} e^{\alpha\sigma}L\big(\gamma(\sigma),\dot \gamma(\sigma)\big)d\sigma \leqslant  \int_{-t}^{0} e^{\alpha\sigma}L\big(\gamma(\sigma),\dot \gamma(\sigma)\big)d\sigma$.
As the function $(q,v)\mapsto H\big(q, \partial_vL(q,v)\big)$ is coercive, there is $R>0$ such that $\|v\|>R$ implies $H\big(q, \partial_vL(q,v)\big)> \max \Big( \frac{\beta}{\alpha}+R_1 ,e^{\alpha t}\big( R_1+\alpha (C+1)\big)\Big)$.
% {\color{blue}$\max ( \frac{\beta}{\alpha}+ R_1+\alpha (C+1) , R_1+\alpha (C+1))$? }
  We conclude that for all $s\in [-t,0]$ we have $\|\dot\gamma(s)\| \leqslant R$ as desired.

The last part of the statement is immediate from the proof. 
\end{proof}

Assume now by contradiction that there are two extremal curves, $\gamma_0,\gamma_1 : [-t,0] \to \R^d$  such that $\gamma_0(-t) = \gamma_1(-t) = q_1$ and $\gamma_0(0) = \gamma_1(0) = q_2$. If $s\in [0,1]$ let $\gamma_s = (1-s)\gamma_0 + s \gamma_1$. Set 
$$C=  \max_{s\in [0,1] } \int_{-t}^0 e^{\alpha \sigma} L\big(\gamma_s(\sigma), \dot{\gamma_s}(\sigma)\big) \ d\sigma.$$

 We then consider $\widetilde{ \mathfrak L}$ that coincides with $\mathfrak L$ on $[-R,R] \times \T^d \times B(0,R)$.
such that   $\widetilde{ \mathfrak L}$ is quadratic at infinity and still Tonelli (such modifications are described in \cite{AF07}). Note that $\gamma_0$ and $\gamma_1$ are still extremal curves of $\widetilde{\mathfrak L}$ and still do not have conjugate points for $ \widetilde{\mathfrak L}$.
 
 We now shift to the functional setting of the Hilbert space $E$ of $H^1$  curves $\eta : [-t, 0] \to \R^d$ such that $\eta(-t) = \eta(0) = 0$ and with a norm $\|\cdot \|_E$ induced by the $H^1$ norm. We focus more precisely on curves in the affine space $\gamma_0+E = \gamma_1+E$. As is well known (see \cite{AF07}) the function 
 $$\mathcal F : \gamma \mapsto \int_{-t}^0 \widetilde{\mathfrak L}\big(s,\gamma(s), \dot{\gamma}(s)\big) \ ds$$
 is $C^2$ on $E$, coercive and verifies the Palais-Smale condition. Moreover, as there are no conjugate points along those curves, $\gamma_0$ and $\gamma_1$ are local strict minima of $\mathcal F$ in the sense that there are $\varepsilon >0$ and $\varepsilon'>0$ such that $\mathcal       F(\gamma_i+\eta)> \mathcal F(\gamma_i)+\varepsilon'$ for $i\in \{0,1\}$ and any $\eta$ of norm $\varepsilon$ (its Hessian is computed in \cite{CI99}  formula (52) for autonomous Lagrangians and the result holds equally in the time-dependent setting).
 
 This allows to use the Ambrosetti-Rabinowitz mountain pass lemma (\cite{S08}, Theorem 6.1 page 109). It asserts that the value 
 $$C' = \inf_\Gamma\max_{s\in [0,1]} \mathcal F\big(\Gamma(s)\big),$$
 where $\Gamma$ ranges over the set of homotopies from $\gamma_0$ to $\gamma_1$, is a critical value of $\mathcal F$. Note that $C' \leqslant C$ by construction of $\widetilde{\mathfrak L}$. More precisely, there exists an extremal curve $\gamma_2 \in E$ such that $\mathcal F(\gamma_2) = C'$ and such that for all $\varepsilon >0$, there exists a homotopy $\Gamma_\varepsilon$ between $\gamma_0$ and $\gamma_1$ such that 
 $$\max_{s\in [0,1]} \mathcal F\big(\Gamma_\varepsilon(s)\big) \leqslant C'+\varepsilon$$
 and $\Gamma_\varepsilon$ intersects the ball centered at $\gamma_2$ and of radius $\varepsilon$. This shows that $\gamma_2$ is not a local strict minimum of $\mathcal F$.
 
 To conclude, we notice that critical points of $\mathcal F$ are $C^2$ trajectories of the Euler-Lagrange flow of $\widetilde{\mathfrak L}$ (\cite{Maz}, Proposition 3.5.1). Hence $\gamma_2$ is $R$-Lipschitz and is also an extremal for $\mathfrak L$ so there are no conjugate points along this trajectory. But this is in contradiction with the fact that it is not a local strict minimum of $\mathcal F$.
 
\end{proof}
\begin{cor}\label{Corimagvert}
Let   $( \varphi_s)$ denote the conformally symplectic Hamiltonian flow on $\R^d$.  Then $\varphi_t(\{q_1\}\times \R^d)$ is a smooth graph above $\R^d$, the graph of $\partial_2\widetilde{\mathcal{A}}_t(q_1,\cdot)$.
\end{cor}
\begin{proof}[Proof of Corollary \ref{Corimagvert}] We denote by  $(\tilde f_s$ the  discounted Euler-Lagrange flow on $\R^d$ of $ L$.  Recall  that $\mathfrak L$ has been defined in Equality \eqref{ELtimedepending}.  It is proven in  Appendix B of \cite{Bernard2008}\footnote{In  \cite{Bernard2008}, there is an extra hypothesis that the Lagrangian is periodic in time. However, this is not used in the proof of the result we use. A way of seing this is to consider $T>t$ and modify $\mathfrak L$ for $s\in [-T,-t]$ to make it $T$ periodic, then we can apply the results of  \cite{Bernard2008} to the modified Lagrangian.} that $\widetilde{\mathcal{A}}_t$ is locally semi-concave  and that for the unique extremal curve $\gamma:[-t, 0]\to \R^d$ such that $\gamma(-t)=q_1$ and $\gamma(0)=q_2$, then 
$$\Big(-\partial_v\mathfrak L\big(-t,q_1, \dot\gamma(-t)\big), \partial_v\mathfrak L\big(0,q_2, \dot\gamma(0)\big)\Big)=\Big(-e^{-\alpha t}\partial_v L\big(q_1, \dot\gamma(-t)\big), \partial_v L\big(q_2, \dot\gamma(0)\big)\Big)$$

 is a super-differential of $\widetilde{\mathcal{A}}_t$ at $(q_1, q_2)$ and that $\tilde f_t \big(q_1, \dot\gamma (-t)\big)=\big(q_2,  \dot\gamma(0)\big)$. This implies that $\varphi_t \Big(q_1,\partial_v L\big(q_1, \dot\gamma(-t)\big)\Big)=\Big(q_2,\partial_v L\big(q_2,  \dot\gamma(0)\big)\Big)$. This is the only orbit whose first projection joins $q_1$ to $q_2$ in time $t$. Hence $\varphi_t\big(\mathcal V(q_1)\big)$ is a graph above $\R^d$. Moreover, being the image by a conformal symplectic dynamics of the  Lagrangian submanifold $\mathcal V(q_1)$, $\varphi_t\big(\mathcal V(q_1)\big)$ is a Lagrangian submanifold that is  as smooth as $ \varphi_t$ is and that is transverse to the fibers because there is no conjugate points. Hence $\varphi_t\big(\mathcal V(q_1)\big)$ is the graph of a function $\nu_{q_1}$ that is as smooth as $\varphi_t$ is.\\
Then we know that for every $q\in\R^d$, $\nu_{q_1}(q)$ is a super-differential of $\widetilde{\mathcal{A}}_t(q_1, \cdot)$, which is locally Lipschitz because it is semi-concave.  A locally Lipschitz function whose almost everywhere differential coincide with a $C^k$ function is $C^{k+1}$. Hence $\widetilde{\mathcal{A}}_t(q_1, \cdot)$ is as differentiable as $H$ is and $\tilde \varphi_t\big(\mathcal V(q_1)\big)=\text{graph}\big(\partial_2\widetilde{\mathcal{A}}_t(q_1, \cdot)\big)$.

\end{proof}
\begin{nota}\rm\begin{enumerate}
\item If $x\in T^*\R^d=\R^d\times \R^d$ and $t\neq 0$, then when defined, 
$$G_t(x)=D\varphi_t\big(\varphi_{-t}(x)\big)V\big(\varphi_{-t}(x)\big)$$
 is transverse to the vertical $V(x)$ because there is no conjugate points. It is then the graph of a $d\times d$ symmetric matrix denoted by $S_t(x)$.
\item  When $S$, $S'$ are two symmetric matrices with the same size, we write $S<S'$ if the matrix $S'-S$ is positive definite.
\end{enumerate}
\end{nota}

\begin{prop}\label{PLipschitzconstant} Assume  $t>0$ and $\varepsilon \in (0, t)$. Then
$S_{-\varepsilon}(q_2, p_2)\leq \partial_2^2\widetilde{\mathcal{A}}_t(q_1, q_2)\leq S_\varepsilon(q_2, p_2)$ where $p_2=\partial_2\widetilde{\mathcal{A}}_t(q_1, q_2)$.
Moreover, this gives a uniform bound for $\partial_2^2\widetilde{\mathcal{A}}_t(q_1, q_2)$ on the set of $(q_1, q_2)$ such that 
$p_2\in [-K, K]^d$. 
\end{prop}

\begin{proof}[Proof of Proposition \ref{PLipschitzconstant}  ] The proof is very close to the proof of Proposition 3.7 in \cite{MCA2008}. We use the notation $x_2=(q_2, p_2)$.

Because there is no conjugate points, $\big(G_\tau(x_2)\big)_{\tau\in \R^*}$ is a family of mutually transverse Lagrangian subspaces of $T_{x_2}T^*\T^d$ that are all transverse to  $V(x_2)$. We deduce that for all $\tau\neq \tau'$ in $\R^*$, the symmetric matrix $S_{\tau'}(x_2)-S_\tau(x_2)$ is non degenerate. Hence its index is constant on each connected set $E_+=\{(\tau, \tau')\in (0,+\infty)^2; 0<\tau<\tau'\}$ and $E_-=(-\infty, 0)\times (0, +\infty)$. To determine these indices, we only need to compute it for some well-chosen pairs  $(\tau_\pm, \tau_\pm')\in E_\pm$.

Let us explain how to choose $(\tau_\pm, \tau_\pm')\in E_\pm$ that are close to $(0, 0)$. For $x\in T^*\T^d$, we introduce a notation for $D\varphi_t(x)$:
$$D\varphi_t(x)=\begin{pmatrix} a_t(x)&b_t(x)\\
c_t(x)&d_t(x)
\end{pmatrix}.$$
The linearized equations imply that
$$\dot a_t(x)=\partial^2_{q, p}H\big(\varphi_t(x)\big);\quad \dot b_t=\partial^2_{p, p}H\big(\varphi_t(x)\big);\quad \dot c_t=-\partial^2_{q, q}H\big(\varphi_t(x)\big);\quad \dot d_t=-\big(\partial ^2_{p, q}H\big(\varphi_t(x)\big)+\alpha {\bf 1}\big).$$
We deduce that $b_t(x)=t\partial^2_{p,p}H\big(\varphi_t(x)\big)+o_{t\to 0}(t)$ and $d_t(x)={\bf 1}+o_{t\to 0}(1)$.

Observe that $G_\tau(x_2)$ is the graph of $S_\tau(x_2)=d_\tau\big(\varphi_{-\tau}(x_2)\big)b_\tau\big(\varphi_{-\tau}(x_2)\big)^{-1}\stackbin[\tau\to 0]{\sim}{}\frac{1}{\tau}\big(\partial^2_{p,p}H(x_2)\big)^{-1}$. Hence for $\tau'=-
\tau$ or $\tau'=2\tau$, we obtain
$$S_{\tau'}(x_2)-S_\tau (x_2)\stackbin[(\tau, \tau')\to (0,0)]{\huge\sim}{}\Big(\frac{1}{\tau'}-\frac{1}{\tau}\Big)\big(\partial^2_{p,p}H(x_2)\big)^{-1}$$
which is negative definite in the two considered cases $\tau'=-\tau<0$ and $\tau'=2\tau>0$. For $\tau>0$, we have  $(-\tau, \tau)\in E_-$ and $(\tau, 2\tau)\in E_+$, hence $S_{\tau'}(x_2)-S_\tau(x_2)$ is positive definite on $E_-$ and positive negative on $E_+$ 

As $(-\varepsilon, t)\in E_-$ and $(\varepsilon, t)\in E_+$, we conclude that $$S_{-\varepsilon}(x_2)<S_t(x_2)=\partial_2^2\widetilde{\mathcal{A}}_t(q_1, q_2)<S_\varepsilon(x_2).$$
Using the continuity of $S_\varepsilon$ and its periodicity in the $q$ variable, we deduce the conclusion on the uniform bound.
\end{proof}

\begin{nota}\rm We denote by $\mathcal K\subset \R^d\times \R^d$ the lift of the global attractor.
\end{nota}
Observe that every point of $\mathcal K$ has a full negative orbit and that the $p$ coordinate is bounded along this orbit. This is even a characterization of the elements of $\mathcal K$.

\begin{prop}\label{PboundnonlinearGreen} There exist two constants $C_1>0$, $C_2>0$ such that 
for every compact and convex  subset $K_0\subset \R^d$, there exists $T=T_{K_0}>0$, such that 
$$\forall x_0\in \mathcal K\cap (K_0\times \R^d), \forall t>T, \quad \varphi_t\Big(\mathcal V\big (\pi \circ \varphi_{-t}(x_0)\big)\Big)\cap (K_0\times \R^d)\subset \R^d\times B(0, C_1)$$
and is the graph of a $C_2$-Lipschitz function from $K_0$ to $B(0, C_1)$.
\end{prop}

\begin{rk}\rm
Observe that we choose $C_1$ and $C_2$ before fixing $K_0$. This will be important for the proof of Proposition \ref{PFinale}.
\end{rk}

\begin{proof}[Proof of Proposition \ref{PboundnonlinearGreen} ] We endow $\R^d$ with its usual Euclidean norm.
Let $(q_0, p_0)\in \mathcal K$, then for all $t$ we have $\big(q_0(t), p_0(t)\big)=\varphi_t(q_0, p_0)\in \mathcal K$. Hence, there exists $C>0$ that  depends only  on $\mathcal K$, such that 
$$\forall s \in \R,\quad  \|\dot q_0(s)\|=\big\|\partial _pH\big(\varphi_s(q_0, p_0)\big)\big\|\leq C.$$
We deduce that 
$$\forall t>0, \quad \Big\| \frac{1}{t}\big(q_0-q_0(-t)\big)\Big\|\leq C.$$
Let now $K_0$ be a compact subset in $\R^d$. We assume that $(q_0, p_0)\in \mathcal K\cap (K_0\times \R^d)$. We    denote $\Delta=\max_{Q, Q'\in K_0}\| Q-Q'\|$.  Let   $q\in K_0$ and $T>\max\{\Delta, 1\}$. We define for  $t\geq T $  and $s\in [-t, 0]$
$$\gamma_q^t(s)=\frac{s+t}{t} q-\frac{s}{t} q_0(-t).$$
Then we have 
$$\|\dot\gamma_q^t(s)\|=\frac{\| q_0(-t)-q\|}{t}\leq \frac{\| q_0(-t)-q_0\|}{t}+\frac{\|q-q_0\|}{t}\leq C+\frac{\Delta}{T} \leq C+1.$$
Then if
$\widetilde C=\frac{1}{\alpha}\max L_{|\R^d\times B(0, C+1)}$, we have

$$\widetilde{\mathcal A}_t(q_0(-t), q)\leq \int_{-t}^0e^{\alpha s} L\big(\gamma_q^t(s), \dot\gamma_q^t(s)\big)ds \leq \alpha\tilde C\int_{-t}^0 e^{\alpha s}ds \leq\widetilde C.$$
Let now $\gamma$ be the unique minimizing curve between $\big(-t, q_0(-t)\big)$ and $(0,q)$. As $t\geq 1$ and $L\geq 0$, we have 

$$\int_{-1}^0 e^{\alpha s}L\big(\gamma(s), \dot\gamma(s)\big)ds\leq \widetilde{\mathcal A}_t(q_0(-t), q)\leq \widetilde C.$$
%{\color{green} Old version : Because $L$ is coercive, this gives a bound (depending only on $\widetilde C$) for $\dot\gamma(s)$ for some $s\in [-1, 0]$, and then, using the flow,  a bound for $\dot\gamma(0)$. }
{ Using Lemma \ref{aprioricompactness} we obtain} a bound for $\dot\gamma(0)$, only depending on $\widetilde C$.
This gives also a bound $C_1$ for $\partial_vL\big(\gamma(0), \dot\gamma (0)\big)=\partial_2\widetilde{\mathcal A}_t(q_0(-t), q)$. We use Proposition \ref{PLipschitzconstant} and Corollary \ref{Corimagvert} to conclude.
\end{proof}
\begin{rk}\rm
In the proof of Proposition \ref{PboundnonlinearGreen}, we use a significant difference between the discounted case and the usual symplectic Hamiltonian case~: for every $A>0$ and every $B>0$, there exists $C>0$ such that for every   $\gamma:(-\infty, A]\to\R^d$ such that $\|\dot\gamma\|\leq B$, then $\int_{-\infty}^A e^{\alpha s}L\big(\gamma(s), \dot\gamma(s)\big)ds\leq C$.
\end{rk} 
 By Ascoli's Theorem and Proposition \ref{PboundnonlinearGreen}, for every orbit $x(t)=\big(q_0(t), p_0(t)\big)=\big(\varphi_t(q_0, p_0)\big)$ of a point $(q_0, p_0)$ of $\mathcal K$, one can obtain convergent  sequences of $\partial_2 \widetilde{\mathcal{A}}_{t_k}(q_0(-t_k), \cdot)$ for the compact-open topology  %for the topology of uniform convergence on compact subsets of $\R^d$ 
 with $t_k$ going to infinity.
\begin{nota}\rm
Let  $\Omega$ be the set of   all the limit points  of $\partial_2 \widetilde{\mathcal{A}}_t(q_0(-t), \cdot)$  when $t\to+\infty$ for the compact-open topology for every $(q_0, p_0)\in \mathcal K$.
\end{nota}
Then  $\Omega \subset C^0\big(\mathbb{R}^d, B(0, C_1)\big)$.  Let us prove the invariance of $\Omega$ by the action of  the flow. Take a small $\tau$.  Then we have
 $$\varphi_t\Big(\mathcal V\big( \pi \circ \varphi_{-t}\circ x(\tau)\big)\Big)=\varphi_\tau \circ \varphi_{t-\tau}\Big(\mathcal V\big(\pi\circ\varphi_{-(t-\tau)}\circ x(0)\big)\Big)$$
 hence 
$$\text{graph}\Big(\partial_2\widetilde{\mathcal A}_t\big(\pi\circ\varphi_{-t}\big(x(\tau)\big), \cdot\big)\Big)=\varphi_\tau\Big(\text{graph}\big(\partial_2\widetilde{\mathcal A}_{t-\tau}( x(0), \cdot)\big)\Big).$$
So when $g\in\Omega$, then for every $t\in\R$, $\varphi_t\big(\text{graph}(g)\big)$ is also the graph of a function of $\Omega$. We will use the notation $\varphi_t\big(\text{graph}(g)\big)=\text{graph}\big(\Phi_t(g)\big)$. Hence $(\Phi_t)$ is a flow on $\mathcal G=\{\text{graph}(g), g\in \Omega\}$. We deduce that the set 
\[
\left\{ \big(x,g(x)\big)~:~ x\in \mathbb{R}^d, g\in \Omega \right\}
\]
is invariant under the conformally Hamiltonian flow $\varphi_t$. 
\begin{prop}\label{PFinale} Let $u_\alpha$ be the discounted weak KAM solution. Then $\Omega=\{ Du_\alpha\}$ and the graph of $Du_\alpha$ is contained in $\mathcal K$.
\end{prop}
%{\color{green}OLD VERSION of Proposition \ref{PFinale} The graphs of the limit points of $(\partial_{2}\widetilde{\mathcal{A}}_t(q_0(-t), \cdot))$ when $t\to+\infty$ for the compact-open topology are {\color{red} included } in $\mathcal K$, and there is only one such graph,   the graph of $du_\alpha$, where $u_\alpha$ is the discounted weak KAM solution.
%}

 \begin{rk}\rm
The fact that the graph of $Du_\alpha$ is contained in $\mathcal K$ is not surprising. It is proven in \cite{AHV} that this graph is contained in the so-called Birkhoff attractor that is contained in $\mathcal K$.
\end{rk}
\begin{proof}[Proof of Proposition \ref{PFinale}] We  begin by proving that every element of $\Omega$  is $\Z^d$-periodic  and even is an exact one-form on $\T^d$. 

We fix $(q_0, p_0)\in \mathcal K$ and denote $\big(q_0(t), p_0(t)\big)=\varphi_t(q_0, p_0)$. Let $(t_k)$ be a sequence of positive numbers that tends to $+\infty$ such that $ \partial_2\widetilde{\mathcal{A}}_{t_k}(q_0(-t_k), \cdot)$ converges for the compact-open topology to some function.  Observe that $ \widetilde{\mathcal{A}}_{t_k}(q_0(-t_k), q_0)=\int_{-t_k}^0e^{\alpha t}L\big(q_0(t), \dot q_0(t)\big)dt$ converges when $k\to \infty$ because $(q_0, p_0)\in \mathcal K$ and then $t\in\R\mapsto \dot q_0(t)$ is bounded. We deduce that $ \widetilde{\mathcal{A}}_{t_k}(q_0(-t_k), \cdot)$ also converges for the compact-open topology when $k$ tends to $\infty$. Then the limit of $ \widetilde{\mathcal{A}}_{t_k}(q_0, \cdot)$ is a primitive of the limit of $ \partial_2\widetilde{\mathcal{A}}_{t_k}(q_0(-t_k), \cdot)$ and then if we show that the limit of $ \widetilde{\mathcal{A}}_{t_k}(q_0(-t_k), \cdot)$ is $\Z^d$-periodic, it is also the case for the limit of the derivatives and they are exact as one-forms on $\T^d$. We use the notation $u$ for the limit of $ \widetilde{\mathcal{A}}_{t_k}(q_0(-t_k), \cdot)$.

Let us fix $\kappa\in\Z^d$. Let $K\subset \R^d$ be a  compact  subset that contains $q_0$. We want to prove that for $q\in K$, we have 
$u(q)=u(q+\kappa)$.  We fix $\varepsilon >0$ and we prove that for $t$ large enough, we have $\widetilde{\mathcal A}_t(q_0(-t), q+\kappa)\leq \widetilde{\mathcal A}_t(q_0(-t), q)+\varepsilon$. Reversing $q$ and $q+\kappa$, this will give the wanted result.

Proposition \ref{PboundnonlinearGreen} yields two constants $C_1$, $C_2$. We deduce a constant $C>0$ such that $\partial_pH\big(\R^d\times B(0, C_1)\big)\subset \R^d\times B(0, C)$ and introduce $C'=\max\{ L(q', v'); \  \| v'\|\leq \| \kappa\|+C_1\}$. 
We fix a large $\tau>0$ such that $C'e^{-\alpha\tau}<\varepsilon$. We denote
$$K_0=\text{convex hull}\Big(\pi\Big( \bigcup_{s\in [0, \tau+1]}\varphi_{-s}\big(K\times B(0, C_1)\big)\Big)\Big).$$
Using Proposition  \ref{PboundnonlinearGreen}, we associate $T$ to $K_0$.\\
Let $\gamma:[-t, 0]\to \R^d$ be the minimizer between $q_0(-t)$ and $q$ for some $t\geq  T + \tau+1$. Then 
$$\Big(q, \partial _vL\big(q, \dot\gamma (0)\big)\Big)=\big(q, \partial_{2}\widetilde{\mathcal A}_{t }(q_0(-t), q)\big)\in \varphi_{t }\Big(\mathcal V\big( \pi\circ\varphi_{-t}(q_0, p_0)\big)\Big)\cap (K\times \R^d)\subset K\times B(0, C_1).$$

Hence for $s\in [0, \tau+1]$, 
\begin{multline*}
\Big(\gamma (-s), \partial _vL\big(\gamma (-s), \dot\gamma (-s)\big)\Big)=\\
\Big(\gamma (-s), \partial_{2}\widetilde{\mathcal A}_{t-s}(q_0(-t), \gamma(-s)\big)\Big)\in \varphi_{t-s}\Big(\mathcal V\big( \pi\circ \varphi_{-(t-s)}\circ\varphi_{-s}(q_0, p_0)\big)\Big)\cap (K_0\times \R^d)
\end{multline*}
and thus we have $\|\dot\gamma (-s)\|=\Big\| \partial_p H\Big( \gamma (-s),\partial _vL\big(\gamma (-s), \dot\gamma (-s)\big)\Big)\Big\|\leq C$.

We denote by $\eta:[-\tau-1, -\tau]\to \R^d$ the geodesic between $\gamma(-\tau-1)$ and $\gamma (-\tau)+\kappa$  , then $\dot\eta=\gamma (-\tau)-\gamma(-\tau-1)+\kappa$ and thus $\|\dot\eta(s)\|\leq \| \kappa\|+C$. We have for $q\in K$
\[\begin{split}\widetilde{\mathcal A}_t(q_0(-t), q+\kappa) & \leq \int_{-t}^{-\tau-1} e^{\alpha s}L\big(\gamma(s), \dot\gamma(s)\big)ds+\int_{-\tau-1}^{-\tau} e^{\alpha s}L\big(\eta(s), \dot\eta(s)\big)ds+\int_{-\tau}^{0} e^{\alpha s}L\big(\gamma(s)+\kappa , \dot\eta(s)\big)ds 
\\
& =\int_{-t}^{-\tau-1} e^{\alpha s}L\big(\gamma(s)+\kappa, \dot\gamma(s)\big)ds+\int_{-\tau-1}^{-\tau} e^{\alpha s}L\big(\gamma(s)+\kappa, \dot\gamma(s)\big)ds+\int_{-\tau}^{0} e^{\alpha s}L\big(\gamma(s)+\kappa , \dot\eta(s)\big)ds 
\\
&\qquad +\int_{-\tau-1}^{-\tau} e^{\alpha s}L\big(\eta(s), \dot\eta(s)\big)ds-\int_{-\tau-1}^{-\tau} e^{\alpha s}L\big(\gamma(s)+\kappa, \dot\gamma(s)\big)ds
\\
 &= \widetilde{\mathcal A}_t(q_0(-t)+\kappa, q+\kappa)+\int_{-\tau-1}^{-\tau} e^{\alpha s}L\big(\eta(s), \dot\eta(s)\big)ds -\int_{-\tau-1}^{-\tau}  e^{\alpha s}L\big(\gamma(s)+\kappa, \dot\gamma(s)\big)ds\\
&\leq \widetilde{\mathcal A}_t(q_0(-t), q)+\int_{-\tau-1}^{-\tau} e^{\alpha s}L\big(\eta(s), \dot\eta(s)\big)ds\leq \widetilde{\mathcal A}_t(q_0(-t), q)+C'e^{-\alpha\tau}.\end{split}\]

Because $C'e^{-\alpha\tau}<\varepsilon$, this gives  the wanted conclusion.
%{\color{blue} METTRE UN DESSIN AVEC LES DIFFERNETS ARCS?}

So the elements of $\Omega$ are in fact  exact one-forms that are defined on $\T^d$. Moreover, the union of their graphs is an invariant set that is contained in the compact set $\T^d\times B(0, C_1)$.   Hence, it is contained in the global attractor.

 Let us prove that $\Omega$ contains a unique element.  Assume that $g_1=du_1, g_2=du_2\in\Omega$. Then for every $t\in\R$, $\text{graph}\big(\Phi_t(g_1)\big)$ and $\text{graph}\big(\Phi_t(g_2)\big)$ are two  exact Lipschitz Lagrangian graphs.  %It is proven in  \cite{ArnFej2021} that they have to be exact Lagrangian graphs, because they are contained in the  global attractor. 
 We will denote by $u_j^t$ a primitive of $\Phi_t(g_j)$.  If $g_1\neq g_2$, then $u_1-u_2$ is not a constant and $\Delta=\max(u_1-u_2)-\min(u_1-u_2)$ is positive. Let $q_m, q_M\in \T^d$ be two points such that $(u_1-u_2)(q_m)=\min (u_1-u_2)$ and $(u_1-u_2)(q_M)=\max(u_1-u_2)$. Let $\gamma:[0, 1]\to \T^d$ be the segment such that $\| \dot \gamma\|\leq \sqrt{d}$, $\gamma(0)=q_m$ and $\gamma (1)=q_M$. Let $\eta: [0, 2]\to T^*\T^d$ the loop defined by
 $$\forall s\in [0, 1],\quad  \eta(s)=\Big(\gamma(s), D u_1\big(\gamma(s)\big)\Big)\text{ and }\forall s\in [1, 2], \quad \eta (s)=\Big(\gamma(2-s), D u_2\big(\gamma(2-s)\big)\Big).$$
 Then $\Delta=\int_\eta pdq$ is the integral of the tautological  1-form along $\eta$. As the flow is conformally symplectic, we have that 
 $$\int_{\varphi_t\circ\gamma}pdq=e^{-\alpha t}\Delta\stackbin[t\to-\infty]{\longrightarrow}{}+\infty.$$
 We introduce $(q_t, p_t)=\varphi_t\big(q_m, D u_1(q_m)\big)\in \text{graph}\big(\Phi_t(g_1)\big)\cap\text{graph}\big(\Phi_t(g_2)\big)$ and 
 $$(Q_t, P_t)=\varphi_t\big(q_M, Du_1(q_M)\big)\in \text{graph}\big(\Phi_t(g_1)\big)\cap\text{graph}\big(\Phi_t(g_2)\big).$$
  Then, because the graphs of $Du_j^t$ are exact Lagrangian, 
 $$\int_{\varphi_t\circ\gamma}pdq=u_1^t(Q_t)-u_2(Q_t)-\big(u_1(q_t)-u_2(q_t)\big).$$
 
 %It is also proven that if $\gamma$ denotes the $\gamma$-distance, then $$\gamma(\text{graph}(\Phi_t(g_1)), \text{graph}(\Phi_t(g_2)))=e^{\alpha t}\gamma(\text{graph}(g_1), \text{graph}(g_2)).$$
%But we know that $\gamma(\text{graph}(\Phi_t(g_1)), \text{graph}(\Phi_t(g_2)))=\max (u_2^t-u_1^t)-\min (u_2^t-u_1^t) $. If $g_1\neq g_2$,   this quantity tends to $+\infty$ when $t$ goes to $-\infty$.  
Then for every $r>0$, there exists $t<0$ and $q\in \T^d$ such that $\| Du^t_2(q)-Du^t_1(q)\|\geq r$, which contradicts the fact that $\text{graph}(Du_1^t)\cup \text{graph}(Du_2^t)$ is a subset of the (compact) global attractor.  So finally $\Omega$ contains only one element, whose graph is then invariant.

As the graph of this element is an exact Lagrangian  Lipschitz graph that is invariant by the flow,  by Remark \ref{rkexactLagweakKAM}, this is the graph of  the differential $Du_\alpha$  of the discounted weak KAM solution and then $\Omega=\{ Du_\alpha\}$.

\end{proof}

\begin{cor}\label{CorHopf}  The graph $\text{graph} (Du_\alpha)$ of the differential of the discounted weak KAM solution is equal to the global attractor.
\end{cor}
%{\color{green} OLD VERSION OF THE COROLLARY: There is only one limit point, it is the graph of $du_\alpha$, that is hence the global attractor.
%}
\begin{proof}[Proof of Corollary \ref{CorHopf} ]
 As the graph of $Du_\alpha$ is a compact invariant set, it is included in $\mathcal K$.

 Let us prove the reverse inclusion.  Let $(q_0, p_0)\in \mathcal K$ and let $K$ be a compact subset of $\R^d$ that contains $q_0$. We proved that   $\big(\partial_{2}\widetilde{\mathcal{A}}_t(q_0(-t), \cdot)\big)$ converges uniformly to $du_\alpha$ on $K$. As $q_0\in K$ and $\partial_{2}\widetilde{\mathcal{A}}_t(q_0(-t), q_0)=p_0$, we conclude that $(q_0, p_0)$ belongs to the graph of $du_\alpha$.
\end{proof}
\subsection{Addendum}

We don't need the last proposition to prove the theorem, but it is interesting to know the result.
\begin{prop}\label{PvonvKAMf}
If $(q_0, p_0)\in \mathcal K$, then $\widetilde{\mathcal{A}}_t(\pi\circ \varphi_{-t}(q_0, p_0),\cdot)$ converges to $u_\alpha$.
\end{prop}
\begin{proof}[Proof of Proposition \ref{PvonvKAMf}]
A result of the proof of Proposition \ref{PFinale} is that 
 $\big(\widetilde{\mathcal{A}}_t(\pi\circ \varphi_{-t}(q_0, p_0),\cdot)\big)_{t>0}$ converges for the compact open topology when $t$ tends to $+\infty$. Then its limits is $u_\alpha+\mu$ for some constant $\mu$. But we have 
 $$\lim_{t\to+\infty} \widetilde{\mathcal{A}}_t(\pi\circ \varphi_{-t}(q_0, p_0),q_0)=u_\alpha(q_0).$$
 We deduce $\mu=0$.

\end{proof}

\section{ Proof of Theorem \ref{ThintzeroasymptMasInd}}\label{SProofTH2}
The reader is referred to  \cite{Arnold1972, CGIP2003, ArnaudFlorioRoos2022} for an overview of Maslov index.
\subsection{Some facts about Maslov index}
\phantom{fish}
We denote by $\mathbb L$ the Lagrangian Grassmanian of $T^*M$, that is the manifold whose elements are Lagrangian subspaces of $T(T^*M)$. We denote for $x\in T^*M$, $\mathbb L_x=\{ L\in\mathbb L; \ \ \pi(L)=x\}$.

The singular cycle is   the algebraic submanifold of $\mathbb L$ defined by  
$$\Sigma=\Big\{ L\in\mathbb L; \dim \Big(L\cap V\big(\pi(L)\big)\Big)\geq 1\Big\},$$
 i.e. $L\in\mathbb L$ if $L$ is not transverse to the vertical foliation. \\
To each arc $\Gamma: I=[a, b]\to \mathbb L$, such that   its endpoints $\Gamma(a)$, $\Gamma(b)$  are transverse to the vertical\footnote{We won't repeat that every time we will compute the Maslov index, we will assume that it is defined.}, we can associate its Maslov index $\text{MI}(\Gamma)$ that counts the algebraic number of intersections with the asymptotic cycle when following $I$ in the positive sense, and eventually perturbing $\Gamma$ to put it in general position with respect to $\Sigma$. It is known that if  $\Gamma_1, \Gamma_2:[a, b]\to \mathbb L$ are homotopic with fixed endpoints, then they have the same Maslov index.

When $I=[a, +\infty)$ and when this is defined, the {\sl asymptotic Maslov index } of $\Gamma$ is 
$$\text{AMI}(\Gamma)=\lim_{t\to+ \infty} \frac{1}{t}\text{MI}( \Gamma_{| [a, a+t]}).$$

When $\Gamma:\R/T\Z\to\mathbb L$ is a loop, the  Maslov index $\text{MI}(\Gamma)$ counts also the algebraic number of intersections with the asymptotic cycle when following $\R/T\Z$ in the positive sense, and eventually perturbing $\Gamma$ to put it in general position with respect to $\Sigma$.
\begin{prop}\label{LtransverseMaslovIndex}
Let $\Gamma_1, \Gamma_2:I=[a,b]\to\mathbb L$ be two arcs such that $\pi\circ\Gamma_1=\pi\circ\Gamma_2$ and such that at each $t\in I$, $\Gamma_1(t)$ and $\Gamma_2(t)$ are transverse. Then 
$$\vert {\mathrm{MI}}(\Gamma_1)-\mathrm{MI}(\Gamma_2)\vert \leq 2(d+1).$$
\end{prop}
\begin{proof}[Proof of Proposition \ref{LtransverseMaslovIndex}]  We denote $\gamma(t)=\pi\circ\Gamma_1(t)=\pi\circ\Gamma_2(t)$. Thanks to the hypothesis of transversality, we can  choose continuously a symplectic basis $(e_1^t, \dots, e_d^t, f_1^t, \dots, f_d^t)$ of $T_{\gamma(t)}(T^*M)$ such that $(e_1^t, \dots, e_d^t)\in \Gamma_1(t)$ and $( f_1^t, \dots, f_d^t)\in \Gamma_2(t)$. Let $\varepsilon>0$, let $\eta:\R\to [0, 1]$ be a smooth function equal to $0$ on $(-\infty, a+\varepsilon]\cup [b-\varepsilon, +\infty)$ and equal to 1 on $[a+2\varepsilon, b-2\varepsilon]$. For $\theta\in [0, \frac{\pi}{2}]$, we define 
$g_i^t(\theta)=\cos (\eta(t)\theta) e_i^t+\sin(\eta(t)\theta) f_i^t$ and $\Gamma(\theta, t)=\text{Vect}\{ g_1^t, \dots, g_d^t\}$. Then $\big(\Gamma(\theta, t)\big)_{\theta\in [0, \frac{\pi}{2}]}$ is an homotopy with fixed endpoints of elements of $\mathbb L$ such that 
\begin{itemize}
\item 
$\forall t\in [a, b], \quad \Gamma(0, t)=\Gamma_1(t)$;
\item $\forall t\in [a, a+\varepsilon]\cup [b-\varepsilon, b], \  \forall \theta\in[0; \frac{\pi}{2}], \quad \Gamma(\theta, t)=\Gamma_1(t)$:
\item $\forall t\in [a+2\varepsilon, b-2\varepsilon], \quad  \Gamma(\frac{\pi}{2}, t)=\Gamma_2(t)$.
\end{itemize} 
We can choose $\varepsilon>0$ small enough in such a way there exists an Lagrangian subspace $H(t)$ of $T_{\gamma(t)}(T^*M)$ that continuously depends on $t\in [a, a+2\varepsilon]\cup [b-2\varepsilon, b]$ such that $H(t)$ is transverse to $\Gamma_1(t)$, $\Gamma_2(t)$ and the vertical $V\big(\gamma(t)\big)$ at $\gamma(t)$.  Hence for $t\in [a, a+2\varepsilon]\cup [b-2\varepsilon, b]$, there exists a symplectic basis $(E^t_1, \dots, E^t_d, F^t_1, \dots, F^t_d)$ of $T_{\gamma(t)}(T^*M)$ such that $E^t_1, \dots , E^t_d\in V\big(\gamma(t)\big)$, $F^t_1, \dots, F^t_d\in H(t)$ and an operator $S_i^t:V(t)\to H(t)$ whose matrix in the bases $(E^t, F^t)$ is symmetric such that 
$$\forall i\in \{ 1, \dots, d\}, \quad e_i^t=E_i^t+S_1^t(E_i^t)\text{ and } f_i^t=E_i^t+S_2^t(E_i^t).$$
Then for  $t\in [a, a+2\varepsilon]\cup [b-2\varepsilon, b]$ and $\theta\in [0, \frac{\pi}{2}]$, we have
$$\forall i\in \{ 1, \dots, d\}, \quad g_i^t(\theta) =\frac{1}{\sqrt{2}}\big(\sin (\eta(t)\theta+\frac{\pi}{4})\big)E_i+ \big(\cos(\eta(t)\theta)S^t_1+\sin(\eta(t)\theta)S^t_2\big)(E_i)$$
where $\frac{1}{\sqrt{2}}\big(\sin (\eta(t)\theta+\frac{\pi}{4})\big)\neq 0$ and $\big(\cos(\eta(t)\theta)S_1+\sin(\eta(t)\theta)S_2\big)(E_i)\in H(t)$.This implies that $\Gamma(\theta, t)$ is the graph of a symmetric operator, namely 
$$S(\theta, t)= \frac{\sqrt{2}}{\sin (\eta(t)\theta+\frac{\pi}{4})}\big(\cos(\eta(t)\theta)S^t_1+\sin(\eta(t)\theta)S^t_2\big)$$
from $V\big(\gamma(t)\big)$ to $H(t)$. Hence for every $\theta\in [0, \frac{\pi}{2}]$, the Maslov index of $\Gamma(\theta,\cdot):[a, a+2\varepsilon]\to \mathbb L$ (resp. $\Gamma(\theta,\cdot):[b-2\varepsilon,b]\to \mathbb L$) corresponds to the variation of the index of $S(\theta, \cdot )$ on $[a, a+2\varepsilon]$ (resp. on $[b-2\varepsilon, b]$). As the matrix has for size $d$, such a variation is between $-(d+1)$ and $d+1$.  

To conclude, as $\Gamma_1$ and $\Gamma(\frac{\pi}{2}, \cdot)$ are homotopic with fixed endpoints, they have the same Maslov index. As $\Gamma(\frac{\pi}{2}, \cdot)_{|[a, a+2\varepsilon]\cup [b-2\varepsilon, b]}=\Gamma_{2|[a, a+2\varepsilon]\cup [b-2\varepsilon, b]}$, we deduce that 
$$\text{MI}(\Gamma_1)-\text{MI}(\Gamma_2)=\text{MI}\big(\Gamma\big(\frac{\pi}{2}, \cdot\big)_{|[a, a+2\varepsilon]\cup [b-2\varepsilon, b]}\big)-\text{MI}(\Gamma_{2|[a, a+2\varepsilon]\cup [b-2\varepsilon, b]})\in [-2(d+1), 2(d+1)].$$

\end{proof}
\begin{cor}\label{CflowMI}
Let $(\varphi_t)$ be a conformally symplectic flow on $T^*M$ and let $x\in T^*M$. Let $L_1, L_2\in \mathbb L_x$. Then, for every $[a, b]\subset \R$, we have
$$\big\vert \mathrm{MI}(D\varphi_sL_1)_{s\in [a, b]}- \mathrm{MI}(D\varphi_sL_2)_{s\in [a, b]} \big\vert\leq 4(d+1).$$
\end{cor}
\begin{proof}[Proof of Corollary \ref{CflowMI}] Let $L_3\in\mathbb L_x$ that is transverse to $L_1$ and $L_2$. Then for every $t\in[a, b]$, $D\varphi_tL_3$ is transverse to $D\varphi_tL_2$ and $D\varphi_tL_1$. Then  by Proposition \ref{LtransverseMaslovIndex}, we have for every $j\in \{ 1, 2\}$
$$\big\vert \text{MI}(D\varphi_sL_j)_{s\in [a, b]}-\text{MI}(D\varphi_sL_3)_{s\in[a, b]}\big\vert\leq 2(d+1);$$
and then
$$\big\vert \text{MI}(D\varphi_sL_1)_{s\in [a, b]}-\text{MI}(D\varphi_sL_2)_{s\in[a, b]} \big \vert\leq 4(d+1).$$

\end{proof}
What will particularly be interesting for us is the so called {\sl dynamical Maslov index}.
\begin{cor}\label{CDynMasInd}
Let $(\varphi_t)$ be a conformally symplectic flow on $T^*M$ and let $x\in T^*M$. 
Assume that for some $L\in \mathbb L$ such that $\pi(L)=x$ the asymptotic Maslov index $\mathrm{AMI}\big(D\varphi_t(L)\big)_{t\in [0, +\infty]}$ is defined. Then it doesn't depend on the  choice $L\in \mathbb L$ such that $\pi(L)=x$, is called the {\rm dynamical Maslov index} of $x$ and denoted by $\mathrm{DMI}(x)$.
\end{cor} 
\begin{proof}[Proof of Corollary \ref{CDynMasInd}] Let us consider $L_1, L_2\in\mathbb L_x$. 
 Then we know by Corollary  \ref{CflowMI} that for every $t>0$, we have 
$$\big\vert \text{MI}(D\varphi_sL_1)_{s\in [0, t]}-\text{MI}(D\varphi_sL_2)_{s\in[0, t]}\big \vert\leq 4(d+1).$$
Hence
$$\lim_{t\to+\infty}\Big(\frac{1}{t}\text{MI}(D\varphi_sL_1)_{s\in [0, t]}-\frac{1}{t}\text{MI}(D\varphi_sL_2)_{s\in [0, t]}\Big)=0.$$
This implies that if the dynamical Maslov index exists for $L_1\in \mathbb L_x$, then it exists and it is the same for every $L_2\in\mathbb L_x$. \end{proof}
\subsection{The Maslov index along curves of $T\mathcal L$ where $\mathcal L$ is   a $C^1$ Lagrangian submanifold of $T^*M$ that is symplectically  isotopic to the zero section
}\phantom{fish}
When $\mathcal L$ is a $C^1$ Lagrangian submanifold of $T^*M$ that is symplectically  isotopic to the zero section, the Maslov index of every loop $\Gamma: \R/T\Z\to \mathbb L$ such that $\Gamma(t)\subset T\mathcal L$ is zero {as it is homotopic to a   loop contained in the tangent bundle of the zero section. We deduce
\begin{prop}\label{PboundedMI} Let $\mathcal L$ be a symplectically isotopic to the zero section $C^1$ Lagrangian submanifold of $T^*M$. Then there exists a constant $C>0$ such that
$$\forall \Gamma:[a, b]\to\mathbb L, \forall t\in [a, b], \quad \Gamma(t)\in T\mathcal L\Longrightarrow \mathrm{MI}(\Gamma)\in [-C, C].$$
\end{prop}
When $\gamma: I\to \mathcal L$, the notation $\widetilde\gamma^{\mathcal L}:I\to \mathbb L$ is for $\widetilde\gamma^{\mathcal L}(t)= T_{\gamma(t)}\mathcal L$.
\begin{proof}[Proof of Proposition \ref{PboundedMI}] Let us assume that there exists a sequence of arcs $\gamma_n:[0, T_n]\to \mathcal L$ such that the Maslov index of $\widetilde \gamma_n^{\mathcal L}:[0, T_n]\to \mathbb L$, $t\mapsto T_{\gamma_n(t)}\mathcal L$ tends to $+\infty$ when $n$ goes to $\infty$. Up to a subsequence, we can assume that $(x_n)_n=\big(\gamma_n(0)\big)_n$ tends to some $x\in\mathcal L$ and $(y_n)_n=\big(\gamma_n(T_n)\big)_n$ tends to $y\in \mathcal L$.
\begin{lemma}\label{LlocalMI}
For every $z\in\mathcal L$, there exists a connected neighborhood $V_z$ of $z$ in $\mathcal L$ such that, for every $z_1, z_2\in V_z$, and every  $\gamma:[0, 1]\to V_z$ such that $\gamma(0)=z_1$, $\gamma(1)=z_2$ then $\mathrm{MI}(\widetilde\gamma^{\mathcal L})\in [-(d+1), d+1]$.
\end{lemma}
\begin{proof}[Proof of Lemma \ref{LlocalMI}] We can choose a small neighborhood $W_z$ of $z$   and a Lagrangian foliation $\mathcal F$ in $W_z$ that is transverse to $T\mathcal L$ on $V_z=W_z\cap \mathcal L$  and transverse to the vertical foliation on $W_z$. We can then choose local symplectic coordinates $(X, Y)\in I\times U_z\times V_z$ in $W_z$ such that the vertical leaves have equation $Y=Y_0$, the leaves of $\mathcal F$ have for equation $X=X_0$ and $\mathcal L\cap W_z$ is a graph $\big\{ \big(X, DU(X)\big);\ \  X\in U_z\big\}$. Hence if $\gamma:[0, 1]\to V_z$ is an arc, the Maslov index along $\widetilde \gamma^{\mathcal L}$ counts the change of signature of $D^2U$ along $\pi\circ\gamma$, that is then in $[-(d+1), d+1]$.
\end{proof}
Let us choose $N\in\N$ large enough such that
$$\forall n\geq N, \quad x_n\in V_x\text{ and } y_n\in V_y.$$
By Lemma \ref{LlocalMI}, there exist $\alpha_n:[0, 1]\to V_x$ and $\beta_n:[0, 1]\to V_y$ such that 
\begin{itemize}
\item $\alpha_n(0)=x$, $\alpha_n(1)=x_n$, $\beta_n(0)=y_n$ and $\beta_n(1)=y$;
\item $\text{MI}(\widetilde\alpha_n^{\mathcal L})\in [-(d+1), d+1]$ and $\text{MI}(\widetilde\beta_n^{\mathcal L})\in [-(d+1), d+1]$.
\end{itemize}
Let now $\gamma:[0, 1]\to \mathcal L$ be an arc such that $\gamma(0)=y$ and $\gamma(1)=x$. For every $n\geq N$, we define $\eta_n= \alpha_n\cup\gamma_n\cup \beta_n\cup\gamma$. Then $\eta_n$ is a loop of $\mathcal L$ such that
$$\text{MI}(\widetilde \eta^{\mathcal L}_n)=\text{MI}(\widetilde \alpha_n^{\mathcal L})+\text{MI}(\widetilde \gamma_n^{\mathcal L})+\text{MI}(\widetilde \beta_n^{\mathcal L})+\text{MI}(\widetilde \gamma^{\mathcal L}).$$
As $\text{MI}(\widetilde \alpha_n^{\mathcal L})\in [-(d+1), d+1)]$, $\text{MI}(\widetilde \beta_n^{\mathcal L})\in [-(d+1), d+1]$, $\text{MI}(\widetilde \gamma^{\mathcal L})$ is a fixed number and $\lim\limits_{n\to \infty} \text{MI}(\widetilde \gamma^{\mathcal L})=+\infty$, the sequence of loops $(\widetilde \eta_n^{\mathcal L})$ has its Maslov index that tends to $+\infty$ when $n$ goes to $\infty$. But this contradicts the fact that the Maslov index of every loop on $\mathcal L$ is zero.

\end{proof}
 \begin{cor}\label{CboundedMI}
  Let $\mathcal L$ be a symplectically isotopic to the zero section, $C^1$, Lagrangian submanifold of $T^*M$ that is invariant by a conformally symplectic flow $(\varphi_t)_t$. Then there exists a constant $C\in \R$ such that for every $x\in \mathcal L$, for every $L\in \mathbb L_x$ and every $[a, b]\subset \R$, we have 
  $$\mathrm{MI}\big(D\varphi_t(L)\big)_{t\in [a, b]}\in [-C, C].$$
 \end{cor} 
 This corollary is a direct result of Proposition \ref{PboundedMI} and Corollary \ref{CflowMI}.
 \subsection{The case of Hopf integrability} \phantom{bouh}
\begin{prop}\label{PMIHopf}
Let us assume that the conformally symplectic flow $(\varphi_t)_t$ of $T^*M$ has a Lipschitz Lagrangian invariant graph $\mathcal G$. Then for every $x\in \mathcal G$, for every $L\in\mathbb L_x$ and every $[a, b]\subset \R$, $\mathrm{MI}(D\varphi_tL)_{t\in[a, b]}\in [-4(d+1), 4(d+1)]$.
\end{prop} 
\begin{proof}[Proof of proposition \ref{PMIHopf}] Let $\eta$ be the closed 1-form such that $\mathcal G=\text{graph}(\eta)$.
By Rademacher Theorem, $\eta$ is differentiable Lebesgue almost everywhere.\\
 If $q$ is such a point of differentiability, then the graph of the differential of $\eta$ is a Lagrangian subspace that we denote $L\big(q, \eta(q)\big)$ of $T_{\textrm{$\big(q, \eta(q)\big)$}}(T^*M)$ and at $\varphi_t\big(q, \eta(q)\big)$, $\eta$ is differentiable and $D\varphi_tL\big(q, \eta(q)\big)=L\Big(\varphi_t\big(q, d\eta(q)\big)\Big)$. All the Lagrangian subspaces being graphs, they are transverse to the vertical and then for every $[a, b]\subset \R$, $\text{MI}\Big(D\varphi_tL\big(q, \eta(q)\big)\Big)_{t\in[a, b]}=0$.\\
 If $q$ is not a point of differentiability of $\eta$, it is the limit of a sequence of points $q_n$ of differentiability of $\eta$. Because $L$ is Lipschitz, the sequence $\Big(L\big(q_n, \eta(q_n)\big)\Big)_{n\in\N}$ is a sequence of graphs of linear maps that are uniformly bounded. Up to extracting a subsequence, we can assume that $\Big(L\big(q_n, \eta(q_n)\big)\Big)_{n\in\N}$ converge to some Lagrangian subspace $L$ of $T_{\textrm{$\big(q, \eta(q)\big)$}}(T^*M)$ that is a graph. Then for every $t\in \R$, $D\varphi_t(L)$ is also a graph and then the Maslov index of every arc $(D\varphi_tL)_{t\in [a, b]}$ is zero. We conclude with the help of Corollary \ref{CflowMI}.

 \end{proof}

 \subsection{The dynamical Maslov index}\phantom{fish}
 We now provide the proof of Theorem \ref{ThintzeroasymptMasInd} in case of $C^1$ or Hopf integrability with attractor $\mathcal L$. Let $x\in T^*M$ and $\eta>0$ such that $D\varphi_\eta V(x)$ is transverse to the vertical. We know that $\omega(x)\subset \mathcal L$. Let us fix $\varepsilon>0$ and let us prove that there exists $T>0$ such that for every $t\geq T$,  
 $$\frac{1}{t}\text{MI}\big(D\varphi_sV(x)\big)_{s\in [\eta, \eta+t]}\in [-\varepsilon, \varepsilon].$$
Using Corollary \ref{CboundedMI} or Proposition \ref{PMIHopf}, we associate a constant $C>0$ to $\mathcal L$ and $(\varphi_t)$. We then choose $T>0$ such that $\frac{2C}{T}<\varepsilon$. There exists a neighborhood $\mathcal N$ of $\mathcal L$  such that 
$$\forall x\in \mathcal N, \forall L\in\mathbb L_x, \forall t\in [0, T], \quad \text{MI}(D\varphi_sL)_{s\in [0, t]}\in [-2C, 2C].$$
As $\omega(x)\subset \mathcal L$, there exists $\tau\geq T$ such that 
$$\forall t\geq \tau, \quad \varphi_t(x)\in \mathcal N.$$
 Let us now consider $t\geq \tau$. Then there exists $k\in\N$ and $\delta t\in [0, T]$ such that
 $t=\tau+kT+\delta t$. Then if $L\in\mathbb L_x$, we have 
 $$\text{MI}\big(D\varphi_s(L)\big)_{s\in [0, t]}=\text{MI}\big(D\varphi_s(L)\big)_{s\in [0, \tau]}+\sum_{j=0}^{k}c_j$$
 where $c_k=\text{MI}(D\varphi_sL)_{s\in[kT +\tau, t]}\in [-2C, 2C]$ and $c_j=\text{MI}(D\varphi_sL)_{s\in[jT+\tau, (j+1)T+\tau]}\in [-2C, 2C]$ for $0\leq j\leq k-1$. Hence 
 $$\frac{1}{t}\text{MI}\big(D\varphi_s(L)\big)_{s\in [0, t]}=\frac{1}{t}\text{MI}\big(D\varphi_s(L)\big)_{s\in [0, \tau]}+\frac{T}{t}\sum_{j=0}^{k}\frac{c_j}{T}. $$
 We have $\big\vert \frac{c_j}{T}\big\vert <\varepsilon$ and $\frac{T}{t}\leq \frac{1}{k}$ hence
 $$\Big\vert \frac{1}{t}\text{MI}\big(D\varphi_s(L)\big)_{s\in [0, t]}\Big\vert < \Big\vert \frac{1}{t}\text{MI}\big(D\varphi_s(L)\big)_{s\in [0, \tau]}\Big\vert + \frac{(k+1)}{k}\varepsilon.$$
This ends the Proof of Theorem \ref{ThintzeroasymptMasInd}.

As a conclusion let us state a results that is proved in a similar way and applies to arbitrary Hamiltonians:

\begin{prop}
Let $H$ be a Hamiltonian and $\varphi$ its flow.  
Let $x\in T^*M$ such that the positive orbit $\{\varphi_t(x) , \ \ t\geqslant 0\}$ is well defined, bounded and $\omega(x)$ does not have conjugate points. Then $\mathrm{DMI}(x) =0$.
\end{prop}
\begin{proof}[sketch of proof]
Arguing as above, for all $T>0$ there is $\tau>T$ such that for all $t>\tau$  for all $L\in \mathbb L_x$ then  $\mathrm{MI}(D\varphi_tL)_{t\in[t, t+T]}\in [-8(d+1), 8(d+1)]$.

The end of the proof is then the same as the proof of Theorem \ref{ThintzeroasymptMasInd}. 
\end{proof}

\section{ A $C^1$ integrable example  that is not Hopf integrable and thus has  conjugate points}\label{C1notHopf}

It would be nice to keep in mind the example in \cite[Section 7.2]{ArnFej2021} where  the ambient manifold is the cotangent bundle of the circle. This example was inspired by an example of Le Calvez \cite{LeCalvez1988}. We will prove that this example is $C^1$-integrable, that the  global attractor is not a graph and that there are conjugate points on the global attractor. We recall the construction.

Let $\beta>0$ be a positive number and let $\alpha\in(\beta, 2\beta)$. On $T^*\R=\R^2$, let $H$ be the quadratic Tonelli Hamiltonian 
$$H(x, y)=y^2-\beta xy.$$
Consider the sum of the Hamiltonian vector field of $H$ and of $\alpha$ times the Liouville vector field $-y\, \partial_y$:
\begin{equation}\label{EPat1}\begin{cases}
  \dot x=-\beta x+2y\\
  \dot y=(\beta-\alpha) y.
\end{cases}\end{equation}
The matrix of this linear system is $\begin{pmatrix} -\beta & 2 \\
0&\beta-\alpha
\end{pmatrix}.$ Hence  $\begin{pmatrix} 1\\0
\end{pmatrix}$ is an eigenvector for the eigenvalue $-\beta$ and $\begin{pmatrix} 1\\ \beta-\frac{\alpha}{2}\end{pmatrix}$ is an eigenvector for the eigenvalue $\beta-\alpha$. As $\alpha\in(\beta, 2\beta)$, $(0,0)$ is an attracting fixed point and the line $\R\begin{pmatrix} 1\\ 0\end{pmatrix}$ is the strong stable eigenspace.  Every solution that is not contained in an eigenspace is contained in a curve whose equation is
$$x=\frac{2}{2\beta-\alpha}y+K\vert y\vert^\frac{\beta}{\alpha-\beta}$$
where $K\neq0$, and then is not a graph if $x(0).y(0) <0$.
% \marginpar{I am under the impression there is a sign error in the paper with Jacques}

%\input{pasgraphe}

Let us choose two large real numbers $B>A>0$ and let $V:\R\rightarrow [-1, 0]$ be an even function with support in $[-B, B]$ such that $V_{|[-A, A]}=-1$, $V_{ | [-B, -A]}$ is non-increasing and $V_{|[A, B]}$ is non-decreasing. Then we add $V(x)$ to $H(x,y)$ and the equations become
\begin{equation}\label{EPat2}\begin{cases}
  \dot x=-\beta x+2y\\
  \dot y=-V'(x)+(\beta-\alpha) y.
\end{cases}\end{equation}
As the support of $V'$ is in $[-B, -A]\cup [A,B]$, the two vector fields are equal in the complement of $([-B, -A]\cup [A, B])\times \R$. As $V'_{|[-B, -A]}\leq 0$, the orbit on the $x$-axis for $x\leq -B$  is pushed to the half plane $y>0$ and then coincides with an orbit of \eqref{EPat1} which tends to $(0, 0)$. In the same way, the orbit that coincides with the $x$-axis for $x\geq B$ tends to $(0, 0)$  at $+\infty$ with an incursion into the half-plane $y<0$. Hence the union of these two orbits and $\{ (0, 0)\}$ is an invariant curve $\Gamma$ for \eqref{EPat2} that is not a graph.

Now, let us choose $D>C>B$ such that
\begin{equation}\label{Epetitdereta}\frac{2(\alpha-\beta)}{\beta}\log\frac{C}{B}>1.
\end{equation}Let $X:\R\rightarrow \R$ be an odd vector field such that
\begin{itemize}
    \item $\forall x\in [-\frac{D+C}{2}, -B]\cup [B, \frac{C+D}{2}], X(x)=-\beta x$;
    \item $X(-D)=X(D)=0$ and all the derivatives of $X$ are the same at $-D$ and $D$;
    \item on $[-D, -B]$, we have $X'(x)\in [-\beta, \beta)$;
    \item $(-D, -B]$ \big(resp. $[B, D)$\big) is a piece of unstable manifold of the equilibrium $-D$ (resp. $D$).
\end{itemize}
Then $X$ defines also a vector field on the circle ${\mathcal C}_D=[-D, D]/D\sim -D$. Let $H_X$ be the Hamiltonian that is associated to $X$ on $T^*\R=\R^2$ via the Ma\~n\'e construction
$$H_X(x, y)=\frac{1}{2}y\big(y+2X(x)\big).$$

Let us eventually define 
\[\begin{split} K(x,y)&=\big(1-\eta (x)\big)H_X(x,y)+\eta(x)\big(H(x,y)+V(x)\big)\\
&=\frac{1-\eta (x)}{2}y\big(y+2X(x)\big)+\eta(x)\big(y^2-\beta xy+V(x)\big),\end{split}\]
where $\eta:\R\rightarrow [0, 1]$ is an even bump function with support in $[-C, C]$ that is equal to 1 on $[-B, B]$ and such that
\begin{equation}\label{Epetitderetabis}\forall x\in [-C, -B], \quad 0\leq  \eta'(x)<\frac{2(\alpha -\beta)}{-\beta x},\end{equation} this choice being possible because of Equation \eqref{Epetitdereta}. $K$ also defines a Hamiltonian function on the annulus $ {\mathcal C}_D\times \R$ and, since
$$ \partial^2_yK(x, y)=\big(1-\eta(x)\big)+2\eta(x)\geq 1$$
hence $K$ is Tonelli.

Note the following:
\begin{itemize}
    \item $([-D, -B]\cup[B, D])\times\{ 0\}$ is in the zero level of $K$ and then is locally invariant by the Hamiltonian flow of $K$ and also by the conformal Hamiltonian flow $( \partial_y K, - \partial_x K -\alpha y)$;
    \item $K_{|[-B, B]\times \R}=(H+V)_{|[-B, B]\times \R}$.
\end{itemize}
Finally, the vector field $( \partial_y K, - \partial_x K -\alpha y)$ has an invariant curve $\mathcal C$ that is not a graph, which is the union of $([-D, -B]\cup[B, D])\times\{ 0\}$ and the part of $\Gamma$ that is between $x=-B$ and $x=B$.
\bigskip

We shall now prove that
\begin{prop}\label{ExampleJacques}
For the example described above, $\mathcal C$ is the global attractor and then $X_K^\lambda$ is $C^1$-integrable with an attractor that is not a graph.  
\end{prop}
\begin{proof}[Proof of Proposition \ref{ExampleJacques}]
We first need a lemma:
 \begin{lemma} \label{Lcriticalpoints} The only zeroes of $X_K^\lambda=( \partial_y K, - \partial_x K -\alpha y) $ are $(-D, 0)$ and $(0, 0)$.

\end{lemma}
Postponing the proof of Lemma \ref{Lcriticalpoints} to the end of the proof of Proposition \ref{ExampleJacques},  let us explain how we deduce     the proposition.\\
The Poincar\'e-Bendixson Theorem tells us that the $\omega$-limit set of every point contains a zero of the vector field  or is a regular invariant smooth closed curve. There is only one invariant smooth closed curve, which is $\mathcal C$. Indeed, if $\mathcal C'$ is another invariant curve, $\T\times \R\backslash (\mathcal C\cup \mathcal C')$ has at least one bounded connected component $U$. Then $U$ is invariant with finite non-zero area. As the flow decreases the areas, this is not possible. Finally, the $\omega$-limit set of every point contains a point of $\mathcal C$. As $\mathcal C$ is a normally hyperbolic attractor, this implies that the $\omega$-limit set of every point is in $\mathcal C$, and the fact that $\mathcal C$ is a local attractor implies that $\mathcal C$ is the global attractor.
\begin{proof}[Proof of Lemma \ref{Lcriticalpoints}] Let us determine the zeroes of the vector field $X_K^\lambda=( \partial_y K, - \partial_x K -\alpha y) $. As $K$ is Tonelli, the equation $\partial_yK=0$ determines a curve that is a graph, and all the zeroes of the vector field $X_K^\lambda$ are on this curve.  As $X_K^\lambda(x, y)=-X_K^\lambda(-x, -y)$, we just have to determine the zeroes on $[-D, 0]\times \R$. Moreover
\begin{itemize}
\item when $x\in [-D, -C]$, $\eta(x)=0$ and $K(x, y)=H_X(x, y)$, so $\partial_yK(x, y)= y+X(x)$ vanishes if and only if $y=-X(x)$, and $\partial_xK(x, y)+\alpha y=y(X'(x)+\alpha)=-X(x)(X'(x)+\alpha)$ vanishes only when $x=-D$ because $X'(x)+\alpha\geq \alpha-\beta>0$;
\item when $x\in [-C, -B]$, we have $X(x)=-\beta x$ and $K(x, y)=\frac{1-\eta (x)}{2}y(y-2\beta x)+\eta (x)(y^2-\beta xy)$. Hence 
$$\partial_y K(x,y)=\big(1-\eta (x)\big)(y-\beta x)+\eta (x) (2y-\beta x)=y\big(1+\eta(x)\big)-\beta x$$
and 
$$\partial_y K(x,y)=0\Longleftrightarrow y=\frac{\beta }{1+\eta(x)} x.$$
Then 
\[\begin{split}\partial_xK\big(x, y\big)+\alpha y &=\eta'(x)\big(y^2-\beta xy-\frac{y^2}{2}+\beta xy\big)-\big(1-\eta (x)\big)\beta y   -\eta(x)\beta y  +\alpha y   \\
&=y\big( \frac{\eta'(x)}{2}y+\alpha -\beta) =\frac{\beta }{1+\eta(x)} x\Big( \frac{\beta\eta'(x)x}{2\big(1+\eta(x)\big)} +\alpha -\beta\Big)\\
\end{split}
\]
Using Equation \eqref{Epetitderetabis}, we have that
$$\frac{\beta\eta'(x)x}{2\big(1+\eta(x)\big)} +\alpha -\beta\geq \frac{\beta\eta'(x)x}{4} +\alpha -\beta>\frac{\alpha-\beta}{2}>0.
$$
Hence there is no zero of $X$ with $x\in [-C, -B]$.

\item When $x\in [-B, B]$, we have $\eta(x)=1$ and $K(x, y)= H(x, y)+V(x)$. Hence $$\partial_y K(x,y)=0\Longleftrightarrow y=\frac{\beta}{2} x$$
and
$$ \partial_x K(x, \frac{\beta}{2} x)+\alpha \frac{\beta}{2} x=(\alpha-\beta)\frac{\beta}{2}x+V'(x)$$
vanishes only when $x=0$, because $x$ and $V'(x)$ have the same sign. \\
Hence $(0, 0)$ is the only point such that $x\in[-B, B]$ and $X_K^\lambda(x, y)=0$.
\end{itemize}
\end{proof}

\end{proof}

As a consequence of Theorem \ref{ThHopf} the previous example has conjugate points.

 The following proposition allows to localize conjugate points on the attractor.

\begin{prop}\label{ExampleJacquesbis}
For the example described above,  there exist conjugate points on $\mathcal C$.
\end{prop}

%---------------------------------------------------------Xifeng put a tikz-picture here
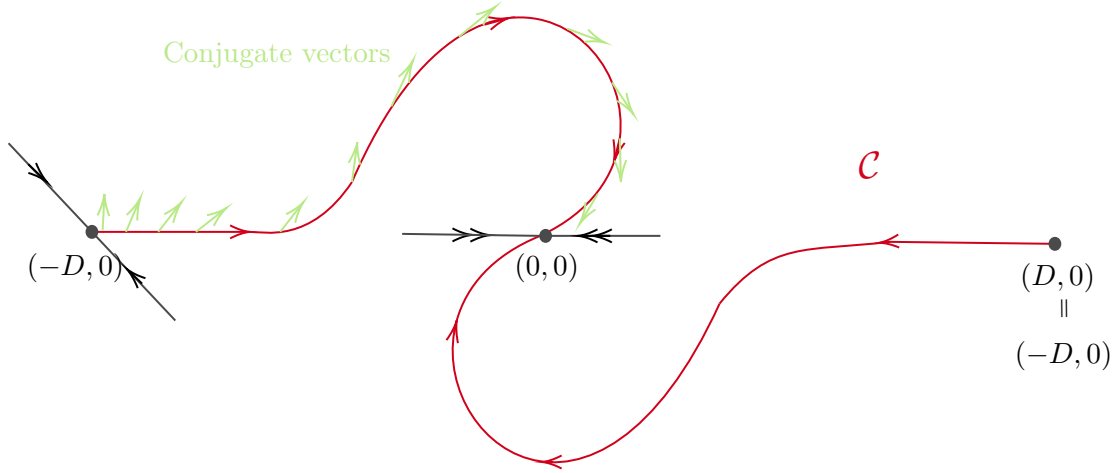
\begin{figure}[H]
\begin{center}
\tikzset{every picture/.style={line width=0.75pt}} %set default line width to 0.75pt        
%\centering
\begin{tikzpicture}[x=0.75pt,y=0.75pt,yscale=-1,xscale=.9]
%uncomment if require: \path (0,258); %set diagram left start at 0, and has height of 258

%Curve Lines [id:da4666919710655044] 
\draw [color={rgb, 255:red, 208; green, 2; blue, 27 }  ,draw opacity=1 ]   (412,159) .. controls (342.01,296.65) and (247.22,226.25) .. (265.45,167.45) .. controls (270.66,150.64) and (285.1,134.78) .. (312,125) ;
%Curve Lines [id:da6576749174267715] 
\draw [color={rgb, 255:red, 208; green, 2; blue, 27 }  ,draw opacity=1 ]   (207,98) .. controls (225.5,59.47) and (246.42,36.68) .. (266.63,25.09) .. controls (276.08,19.68) and (285.37,16.71) .. (294.19,15.73) .. controls (331.83,11.57) and (360.83,43.63) .. (356.38,76.05) .. controls (353.85,94.44) and (340.57,112.93) .. (312,125) ;
%Straight Lines [id:da48626024763164644] 
\draw    (273.99,124.51) ;
\draw [shift={(273.99,124.51)}, rotate = 180] [color={rgb, 255:red, 0; green, 0; blue, 0 }  ][line width=0.75]    (10.93,-3.29) .. controls (6.95,-1.4) and (3.31,-0.3) .. (0,0) .. controls (3.31,0.3) and (6.95,1.4) .. (10.93,3.29)   ;
%Straight Lines [id:da5039462191693809] 
\draw [color={rgb, 255:red, 74; green, 74; blue, 74 }  ,draw opacity=1 ]   (312,125) -- (334,125) -- (341,125) -- (379,125) ;
%Straight Lines [id:da038642308521093405] 
\draw    (283.97,124.64) ;
\draw [shift={(283.97,124.64)}, rotate = 180] [color={rgb, 255:red, 0; green, 0; blue, 0 }  ][line width=0.75]    (10.93,-3.29) .. controls (6.95,-1.4) and (3.31,-0.3) .. (0,0) .. controls (3.31,0.3) and (6.95,1.4) .. (10.93,3.29)   ;
%Straight Lines [id:da6955530896143223] 
\draw [color={rgb, 255:red, 74; green, 74; blue, 74 }  ,draw opacity=1 ]   (235,124) -- (273.99,124.51) -- (283.97,124.64) -- (312,125) ;
%Straight Lines [id:da738834444243256] 
\draw    (351,125) -- (335,125) ;
\draw [shift={(333,125)}, rotate = 360] [color={rgb, 255:red, 0; green, 0; blue, 0 }  ][line width=0.75]    (10.93,-3.29) .. controls (6.95,-1.4) and (3.31,-0.3) .. (0,0) .. controls (3.31,0.3) and (6.95,1.4) .. (10.93,3.29)   ;
%Straight Lines [id:da8955339953957084] 
\draw    (351,125) -- (343,125) ;
\draw [shift={(341,125)}, rotate = 360] [color={rgb, 255:red, 0; green, 0; blue, 0 }  ][line width=0.75]    (10.93,-3.29) .. controls (6.95,-1.4) and (3.31,-0.3) .. (0,0) .. controls (3.31,0.3) and (6.95,1.4) .. (10.93,3.29)   ;
%Straight Lines [id:da9521619673572665] 
\draw [color={rgb, 255:red, 208; green, 2; blue, 27 }  ,draw opacity=1 ]   (507,128) -- (599,129) ;
%Straight Lines [id:da06692027345987217] 
\draw [color={rgb, 255:red, 208; green, 2; blue, 27 }  ,draw opacity=1 ]   (62,123) -- (150,123) ;
%Curve Lines [id:da33180679217262] 
\draw [color={rgb, 255:red, 208; green, 2; blue, 27 }  ,draw opacity=1 ]   (150,123) .. controls (158,123) and (184,129) .. (207,98) ;
%Curve Lines [id:da254628464743446] 
\draw [color={rgb, 255:red, 208; green, 2; blue, 27 }  ,draw opacity=1 ]   (412,159) .. controls (441,127) and (460,133) .. (507,128) ;
%Straight Lines [id:da23448062316650065] 
\draw [color={rgb, 255:red, 208; green, 2; blue, 27 }  ,draw opacity=1 ]   (507,128) -- (502.99,128.33) ;
\draw [shift={(501,128.5)}, rotate = 355.24] [color={rgb, 255:red, 208; green, 2; blue, 27 }  ,draw opacity=1 ][line width=0.75]    (10.93,-3.29) .. controls (6.95,-1.4) and (3.31,-0.3) .. (0,0) .. controls (3.31,0.3) and (6.95,1.4) .. (10.93,3.29)   ;
%Straight Lines [id:da8866673875279696] 
\draw [color={rgb, 255:red, 208; green, 2; blue, 27 }  ,draw opacity=1 ]   (319,238.5) -- (315,238.5) ;
\draw [shift={(313,238.5)}, rotate = 360] [color={rgb, 255:red, 208; green, 2; blue, 27 }  ,draw opacity=1 ][line width=0.75]    (10.93,-3.29) .. controls (6.95,-1.4) and (3.31,-0.3) .. (0,0) .. controls (3.31,0.3) and (6.95,1.4) .. (10.93,3.29)   ;
%Straight Lines [id:da21657261657879667] 
\draw [color={rgb, 255:red, 208; green, 2; blue, 27 }  ,draw opacity=1 ]   (264,173.5) -- (264.98,169.39) ;
\draw [shift={(265.45,167.45)}, rotate = 103.44] [color={rgb, 255:red, 208; green, 2; blue, 27 }  ,draw opacity=1 ][line width=0.75]    (10.93,-3.29) .. controls (6.95,-1.4) and (3.31,-0.3) .. (0,0) .. controls (3.31,0.3) and (6.95,1.4) .. (10.93,3.29)   ;
%Straight Lines [id:da9553576910746826] 
\draw [color={rgb, 255:red, 208; green, 2; blue, 27 }  ,draw opacity=1 ]   (356.38,76.05) -- (353.54,86.08) ;
\draw [shift={(353,88)}, rotate = 285.78] [color={rgb, 255:red, 208; green, 2; blue, 27 }  ,draw opacity=1 ][line width=0.75]    (10.93,-3.29) .. controls (6.95,-1.4) and (3.31,-0.3) .. (0,0) .. controls (3.31,0.3) and (6.95,1.4) .. (10.93,3.29)   ;
%Straight Lines [id:da19545034354025592] 
\draw [color={rgb, 255:red, 208; green, 2; blue, 27 }  ,draw opacity=1 ]   (281,19) -- (292.25,16.21) ;
\draw [shift={(294.19,15.73)}, rotate = 166.09] [color={rgb, 255:red, 208; green, 2; blue, 27 }  ,draw opacity=1 ][line width=0.75]    (10.93,-3.29) .. controls (6.95,-1.4) and (3.31,-0.3) .. (0,0) .. controls (3.31,0.3) and (6.95,1.4) .. (10.93,3.29)   ;
%Straight Lines [id:da29185174533583846] 
\draw [color={rgb, 255:red, 208; green, 2; blue, 27 }  ,draw opacity=1 ]   (150,123) ;
\draw [shift={(150,123)}, rotate = 180] [color={rgb, 255:red, 208; green, 2; blue, 27 }  ,draw opacity=1 ][line width=0.75]    (10.93,-3.29) .. controls (6.95,-1.4) and (3.31,-0.3) .. (0,0) .. controls (3.31,0.3) and (6.95,1.4) .. (10.93,3.29)   ;
%Straight Lines [id:da6756702837668493] 
\draw [color={rgb, 255:red, 74; green, 74; blue, 74 }  ,draw opacity=1 ]   (15.5,78.5) -- (37.48,99.54) -- (81.43,141.59) -- (108.5,167.5) ;
%Straight Lines [id:da6216172115865601] 
\draw    (26.49,89.02) -- (35.47,96.74) ;
\draw [shift={(36.98,98.04)}, rotate = 220.68] [color={rgb, 255:red, 0; green, 0; blue, 0 }  ][line width=0.75]    (10.93,-3.29) .. controls (6.95,-1.4) and (3.31,-0.3) .. (0,0) .. controls (3.31,0.3) and (6.95,1.4) .. (10.93,3.29)   ;
%Straight Lines [id:da41399486496714233] 
\draw    (90,149) -- (82.94,142.9) ;
\draw [shift={(81.43,141.59)}, rotate = 40.83] [color={rgb, 255:red, 0; green, 0; blue, 0 }  ][line width=0.75]    (10.93,-3.29) .. controls (6.95,-1.4) and (3.31,-0.3) .. (0,0) .. controls (3.31,0.3) and (6.95,1.4) .. (10.93,3.29)   ;
%Straight Lines [id:da5055017644320561] 
\draw [color={rgb, 255:red, 184; green, 233; blue, 134 }  ,draw opacity=1 ]   (68,122) -- (68.89,106) ;
\draw [shift={(69,104)}, rotate = 93.18] [color={rgb, 255:red, 184; green, 233; blue, 134 }  ,draw opacity=1 ][line width=0.75]    (10.93,-3.29) .. controls (6.95,-1.4) and (3.31,-0.3) .. (0,0) .. controls (3.31,0.3) and (6.95,1.4) .. (10.93,3.29)   ;
%Straight Lines [id:da08477005889437828] 
\draw [color={rgb, 255:red, 184; green, 233; blue, 134 }  ,draw opacity=1 ]   (99,123) -- (110.69,109.51) ;
\draw [shift={(112,108)}, rotate = 130.91] [color={rgb, 255:red, 184; green, 233; blue, 134 }  ,draw opacity=1 ][line width=0.75]    (10.93,-3.29) .. controls (6.95,-1.4) and (3.31,-0.3) .. (0,0) .. controls (3.31,0.3) and (6.95,1.4) .. (10.93,3.29)   ;
%Straight Lines [id:da6940075326277388] 
\draw [color={rgb, 255:red, 184; green, 233; blue, 134 }  ,draw opacity=1 ]   (81,123) -- (88.15,107.81) ;
\draw [shift={(89,106)}, rotate = 115.2] [color={rgb, 255:red, 184; green, 233; blue, 134 }  ,draw opacity=1 ][line width=0.75]    (10.93,-3.29) .. controls (6.95,-1.4) and (3.31,-0.3) .. (0,0) .. controls (3.31,0.3) and (6.95,1.4) .. (10.93,3.29)   ;
%Straight Lines [id:da31623244589804744] 
\draw [color={rgb, 255:red, 184; green, 233; blue, 134 }  ,draw opacity=1 ]   (120,123) -- (135.37,112.15) ;
\draw [shift={(137,111)}, rotate = 144.78] [color={rgb, 255:red, 184; green, 233; blue, 134 }  ,draw opacity=1 ][line width=0.75]    (10.93,-3.29) .. controls (6.95,-1.4) and (3.31,-0.3) .. (0,0) .. controls (3.31,0.3) and (6.95,1.4) .. (10.93,3.29)   ;
%Straight Lines [id:da3154312516708324] 
\draw [color={rgb, 255:red, 184; green, 233; blue, 134 }  ,draw opacity=1 ]   (167,123) -- (178.69,109.51) ;
\draw [shift={(180,108)}, rotate = 130.91] [color={rgb, 255:red, 184; green, 233; blue, 134 }  ,draw opacity=1 ][line width=0.75]    (10.93,-3.29) .. controls (6.95,-1.4) and (3.31,-0.3) .. (0,0) .. controls (3.31,0.3) and (6.95,1.4) .. (10.93,3.29)   ;
%Straight Lines [id:da42151128519840353] 
\draw [color={rgb, 255:red, 184; green, 233; blue, 134 }  ,draw opacity=1 ]   (207,98) -- (209.7,79.98) ;
\draw [shift={(210,78)}, rotate = 98.53] [color={rgb, 255:red, 184; green, 233; blue, 134 }  ,draw opacity=1 ][line width=0.75]    (10.93,-3.29) .. controls (6.95,-1.4) and (3.31,-0.3) .. (0,0) .. controls (3.31,0.3) and (6.95,1.4) .. (10.93,3.29)   ;
%Straight Lines [id:da7251482404995364] 
\draw [color={rgb, 255:red, 184; green, 233; blue, 134 }  ,draw opacity=1 ]   (266.63,25.09) -- (285.39,11.19) ;
\draw [shift={(287,10)}, rotate = 143.47] [color={rgb, 255:red, 184; green, 233; blue, 134 }  ,draw opacity=1 ][line width=0.75]    (10.93,-3.29) .. controls (6.95,-1.4) and (3.31,-0.3) .. (0,0) .. controls (3.31,0.3) and (6.95,1.4) .. (10.93,3.29)   ;
%Straight Lines [id:da00012492067797120754] 
\draw [color={rgb, 255:red, 184; green, 233; blue, 134 }  ,draw opacity=1 ]   (229,60) -- (240.07,38.77) ;
\draw [shift={(241,37)}, rotate = 117.55] [color={rgb, 255:red, 184; green, 233; blue, 134 }  ,draw opacity=1 ][line width=0.75]    (10.93,-3.29) .. controls (6.95,-1.4) and (3.31,-0.3) .. (0,0) .. controls (3.31,0.3) and (6.95,1.4) .. (10.93,3.29)   ;
%Straight Lines [id:da8819937224185717] 
\draw [color={rgb, 255:red, 184; green, 233; blue, 134 }  ,draw opacity=1 ]   (327,21) -- (346.1,27.37) ;
\draw [shift={(348,28)}, rotate = 198.43] [color={rgb, 255:red, 184; green, 233; blue, 134 }  ,draw opacity=1 ][line width=0.75]    (10.93,-3.29) .. controls (6.95,-1.4) and (3.31,-0.3) .. (0,0) .. controls (3.31,0.3) and (6.95,1.4) .. (10.93,3.29)   ;
%Straight Lines [id:da07512423752012176] 
\draw [color={rgb, 255:red, 184; green, 233; blue, 134 }  ,draw opacity=1 ]   (352,48) -- (362.8,62.4) ;
\draw [shift={(364,64)}, rotate = 233.13] [color={rgb, 255:red, 184; green, 233; blue, 134 }  ,draw opacity=1 ][line width=0.75]    (10.93,-3.29) .. controls (6.95,-1.4) and (3.31,-0.3) .. (0,0) .. controls (3.31,0.3) and (6.95,1.4) .. (10.93,3.29)   ;
%Straight Lines [id:da5309671867756891] 
\draw [color={rgb, 255:red, 184; green, 233; blue, 134 }  ,draw opacity=1 ]   (356.38,76.05) -- (356.94,96) ;
\draw [shift={(357,98)}, rotate = 268.37] [color={rgb, 255:red, 184; green, 233; blue, 134 }  ,draw opacity=1 ][line width=0.75]    (10.93,-3.29) .. controls (6.95,-1.4) and (3.31,-0.3) .. (0,0) .. controls (3.31,0.3) and (6.95,1.4) .. (10.93,3.29)   ;
%Straight Lines [id:da40122673885004634] 
\draw [color={rgb, 255:red, 184; green, 233; blue, 134 }  ,draw opacity=1 ]   (344,104) -- (335.11,117.34) ;
\draw [shift={(334,119)}, rotate = 303.69] [color={rgb, 255:red, 184; green, 233; blue, 134 }  ,draw opacity=1 ][line width=0.75]    (10.93,-3.29) .. controls (6.95,-1.4) and (3.31,-0.3) .. (0,0) .. controls (3.31,0.3) and (6.95,1.4) .. (10.93,3.29)   ;
%Shape: Circle [id:dp0605739311540634] 
\draw  [color={rgb, 255:red, 74; green, 74; blue, 74 }  ,draw opacity=1 ][fill={rgb, 255:red, 74; green, 74; blue, 74 }  ,fill opacity=1 ] (59,123) .. controls (59,121.34) and (60.34,120) .. (62,120) .. controls (63.66,120) and (65,121.34) .. (65,123) .. controls (65,124.66) and (63.66,126) .. (62,126) .. controls (60.34,126) and (59,124.66) .. (59,123) -- cycle ;
%Shape: Circle [id:dp791792433599275] 
\draw  [color={rgb, 255:red, 74; green, 74; blue, 74 }  ,draw opacity=1 ][fill={rgb, 255:red, 74; green, 74; blue, 74 }  ,fill opacity=1 ] (312,125) .. controls (312,123.34) and (313.34,122) .. (315,122) .. controls (316.66,122) and (318,123.34) .. (318,125) .. controls (318,126.66) and (316.66,128) .. (315,128) .. controls (313.34,128) and (312,126.66) .. (312,125) -- cycle ;
%Shape: Circle [id:dp3090685345971902] 
\draw  [color={rgb, 255:red, 74; green, 74; blue, 74 }  ,draw opacity=1 ][fill={rgb, 255:red, 74; green, 74; blue, 74 }  ,fill opacity=1 ] (596,129) .. controls (596,127.34) and (597.34,126) .. (599,126) .. controls (600.66,126) and (602,127.34) .. (602,129) .. controls (602,130.66) and (600.66,132) .. (599,132) .. controls (597.34,132) and (596,130.66) .. (596,129) -- cycle ;

% Text Node
\draw (100,26) node [anchor=north west][inner sep=0.75pt]   [align=left] {\textcolor[rgb]{0.72,0.91,0.53}{Conjugate vectors}};
% Text Node
\draw (487.94,81.36) node [anchor=north west][inner sep=0.75pt]  [font=\Large,rotate=-0.32,xslant=-0.01]  {$\mathcal{\textcolor[rgb]{0.82,0.01,0.11}{C}}$};
% Text Node
\draw (24,132.4) node [anchor=north west][inner sep=0.75pt]    {$( -D,0)$};
% Text Node
\draw (578,137.4) node [anchor=north west][inner sep=0.75pt]    {$( D,0)$};
% Text Node
\draw (296,132.4) node [anchor=north west][inner sep=0.75pt]    {$( 0,0)$};
% Text Node
\draw (607.46,153.91) node [anchor=north west][inner sep=0.75pt]  [rotate=-89.33]  {$=$};
% Text Node
\draw (575,176.4) node [anchor=north west][inner sep=0.75pt]    {$( -D,0)$};

\end{tikzpicture}
\end{center}
\end{figure}

\begin{proof}[Proof of Proposition \ref{ExampleJacquesbis}]
We consider the hyperbolic saddle zero $(-D, 0)$ of $X_K^\lambda$. Recall that close to this point, we have $X_K^\lambda (x, y)=\frac{1}{2}y\big(y+2X(x)\big)$, hence locally the unstable manifold is the $x$-axis and the stable manifold has for equation $y=-2X(x)$ and has a negative slope. Hence if we consider the vertical vector $(0, 1)$ at a point $(-D+\varepsilon, 0)$ on $\mathcal C$. If $\varepsilon>0$ is small enough, because of the $\lambda$-lemma, $D\varphi_T(-D+\varepsilon, 0)(0, 1)$ is very close to $T\mathcal C$ and then there will be a $T>0$ such that  $D\varphi_T(-D+\varepsilon, 0)(0, 1)$ is vertical, see the picture above.

\end{proof}

\section{An example of a Hopf integrable Hamiltonian with a non $C^1$ attractor}\label{HopfnotC1}

We aim at constructing a quadratic Hamiltonian of the form $H(x,p)= \frac 12 p^2 + a(x)p+b(x)$ on $T^*\T^1$ that has the red curve (in the picture below) as global attractor.

\begin{figure}[h!]
\begin{center}
\tikzset{every picture/.style={line width=0.75pt}} %set default line width to 0.75pt        
%\centering
\begin{tikzpicture}[x=0.75pt,y=0.75pt,yscale=-1,xscale=1]
%uncomment if require: \path (0,330); %set diagram left start at 0, and has height of 330

%Straight Lines [id:da6352280123566536] 
\draw [color={rgb, 255:red, 208; green, 2; blue, 27 }  ,draw opacity=1 ][line width=0.75]    (342,154.5) -- (385,154.5) ;
\draw [shift={(358.3,154.5)}, rotate = 0] [color={rgb, 255:red, 208; green, 2; blue, 27 }  ,draw opacity=1 ][line width=0.75]    (7.65,-2.3) .. controls (4.86,-0.97) and (2.31,-0.21) .. (0,0) .. controls (2.31,0.21) and (4.86,0.98) .. (7.65,2.3)   ;
%Straight Lines [id:da07370807697446757] 
\draw [color={rgb, 255:red, 208; green, 2; blue, 27 }  ,draw opacity=1 ]   (385,154.5) -- (427,155) ;
\draw [shift={(410.2,154.8)}, rotate = 180.68] [color={rgb, 255:red, 208; green, 2; blue, 27 }  ,draw opacity=1 ][line width=0.75]    (7.65,-2.3) .. controls (4.86,-0.97) and (2.31,-0.21) .. (0,0) .. controls (2.31,0.21) and (4.86,0.98) .. (7.65,2.3)   ;
%Shape: Circle [id:dp3732597499907384] 
\draw  [fill={rgb, 255:red, 0; green, 0; blue, 0 }  ,fill opacity=1 ] (383.5,154.5) .. controls (383.5,153.67) and (384.17,153) .. (385,153) .. controls (385.83,153) and (386.5,153.67) .. (386.5,154.5) .. controls (386.5,155.33) and (385.83,156) .. (385,156) .. controls (384.17,156) and (383.5,155.33) .. (383.5,154.5) -- cycle ;
%Straight Lines [id:da7201537618152964] 
\draw [color={rgb, 255:red, 208; green, 2; blue, 27 }  ,draw opacity=1 ]   (190,153.5) -- (342,154.5) ;
\draw [shift={(260.8,153.97)}, rotate = 0.38] [color={rgb, 255:red, 208; green, 2; blue, 27 }  ,draw opacity=1 ][line width=0.75]    (7.65,-2.3) .. controls (4.86,-0.97) and (2.31,-0.21) .. (0,0) .. controls (2.31,0.21) and (4.86,0.98) .. (7.65,2.3)   ;
%Straight Lines [id:da597013988158108] 
\draw [color={rgb, 255:red, 208; green, 2; blue, 27 }  ,draw opacity=1 ]   (164,171) -- (190,153.5) ;
\draw [shift={(180.48,159.9)}, rotate = 146.06] [color={rgb, 255:red, 208; green, 2; blue, 27 }  ,draw opacity=1 ][line width=0.75]    (7.65,-2.3) .. controls (4.86,-0.97) and (2.31,-0.21) .. (0,0) .. controls (2.31,0.21) and (4.86,0.98) .. (7.65,2.3)   ;
%Shape: Circle [id:dp6915745430360551] 
\draw  [fill={rgb, 255:red, 0; green, 0; blue, 0 }  ,fill opacity=1 ] (188.25,153.48) .. controls (188.26,152.51) and (189.06,151.74) .. (190.02,151.75) .. controls (190.99,151.76) and (191.76,152.56) .. (191.75,153.52) .. controls (191.74,154.49) and (190.94,155.26) .. (189.98,155.25) .. controls (189.01,155.24) and (188.24,154.44) .. (188.25,153.48) -- cycle ;
%Straight Lines [id:da47926089949051187] 
\draw [color={rgb, 255:red, 208; green, 2; blue, 27 }  ,draw opacity=1 ]   (556.02,169.75) -- (581,155.5) ;
\draw [shift={(572.16,160.54)}, rotate = 150.29] [color={rgb, 255:red, 208; green, 2; blue, 27 }  ,draw opacity=1 ][line width=0.75]    (7.65,-2.3) .. controls (4.86,-0.97) and (2.31,-0.21) .. (0,0) .. controls (2.31,0.21) and (4.86,0.98) .. (7.65,2.3)   ;
%Straight Lines [id:da14203331632541838] 
\draw    (427,155) -- (581,155.5) ;
%Curve Lines [id:da4450300210826097] 
\draw [color={rgb, 255:red, 208; green, 2; blue, 27 }  ,draw opacity=1 ]   (461,140.5) .. controls (523,58.5) and (462,228.5) .. (556.02,169.75) ;
%Straight Lines [id:da8298546245782212] 
\draw  [dash pattern={on 0.84pt off 2.51pt}]  (419,78.25) -- (419,233.75) ;
%Straight Lines [id:da957873681097156] 
\draw  [dash pattern={on 0.84pt off 2.51pt}]  (350,77.75) -- (350,233.25) ;
%Straight Lines [id:da2123455326802608] 
\draw  [dash pattern={on 0.84pt off 2.51pt}]  (224,78) -- (224,233.5) ;
%Straight Lines [id:da41015880979670793] 
\draw  [dash pattern={on 0.84pt off 2.51pt}]  (547,77) -- (547,232.5) ;
%Shape: Circle [id:dp0070646373135223595] 
\draw  [fill={rgb, 255:red, 0; green, 0; blue, 0 }  ,fill opacity=1 ] (579.25,155.48) .. controls (579.26,154.51) and (580.06,153.74) .. (581.02,153.75) .. controls (581.99,153.76) and (582.76,154.56) .. (582.75,155.52) .. controls (582.74,156.49) and (581.94,157.26) .. (580.98,157.25) .. controls (580.01,157.24) and (579.24,156.44) .. (579.25,155.48) -- cycle ;
%Straight Lines [id:da5514251819044182] 
\draw    (385,154.5) -- (403,173) ;
\draw [shift={(390.37,160.02)}, rotate = 45.78] [color={rgb, 255:red, 0; green, 0; blue, 0 }  ][line width=0.75]    (7.65,-2.3) .. controls (4.86,-0.97) and (2.31,-0.21) .. (0,0) .. controls (2.31,0.21) and (4.86,0.98) .. (7.65,2.3)   ;
%Straight Lines [id:da9501049052057069] 
\draw    (367,136) -- (385,154.5) ;
\draw [shift={(378.93,148.26)}, rotate = 225.78] [color={rgb, 255:red, 0; green, 0; blue, 0 }  ][line width=0.75]    (7.65,-2.3) .. controls (4.86,-0.97) and (2.31,-0.21) .. (0,0) .. controls (2.31,0.21) and (4.86,0.98) .. (7.65,2.3)   ;
%Curve Lines [id:da9491584670598074] 
\draw [color={rgb, 255:red, 208; green, 2; blue, 27 }  ,draw opacity=1 ]   (427,155) .. controls (439,153.5) and (451,154.5) .. (461,140.5) ;

% Text Node
\draw (360.57,164.84) node [anchor=north west][inner sep=0.75pt]  [font=\tiny,rotate=-358.91]  {$\left(\frac{1}{2} ,0\right)$};
% Text Node
\draw (179,165.65) node [anchor=north west][inner sep=0.75pt]  [font=\tiny]  {$( 0,0)$};
% Text Node
\draw (570.5,166.9) node [anchor=north west][inner sep=0.75pt]  [font=\tiny]  {$( 1,0)$};
% Text Node
\draw (127,45.4) node [anchor=north west][inner sep=0.75pt]  [font=\scriptsize]  {$H_{1} =\frac{1}{2}\left( p^{2} -\beta \ p\ x\right)$};
% Text Node
\draw (327,44.4) node [anchor=north west][inner sep=0.75pt]  [font=\scriptsize]  {$H_{2} =\frac{\alpha }{2} p( x-x_{0}) +\ \frac{1}{2} p^{2}$};

\end{tikzpicture}
\end{center}
\end{figure}

Let us start from a similar Hamiltonian as previously, $H_1(x,p) = \frac 12 (p^2-\beta px)$.  For a given conformal factor $\alpha>0$ the conformally symplectic (linear) vectorfield of $H_1$ is $$X_1(x,p) = 
\begin{pmatrix}
 p-\frac 12 \beta x \\ (\frac12 \beta -\alpha)p
 \end{pmatrix}
 $$ 

It has the point $(0,0)$ as only fixed point with eigenvalues $-\frac{\beta}{2}$ and $\frac{\beta}{2}-\alpha$ and respective eigenvectors
 $$V_{-\frac{\beta}{2}}=
\begin{pmatrix}
1\\0
\end{pmatrix}
;\qquad
V_{\frac{\beta}{2}-\alpha}=
\begin{pmatrix}
1\\\beta-\alpha
\end{pmatrix}
.
$$
We take $\frac{\beta}{2} < \alpha < \beta$, the first inequality guaranties that $(0,0)$ is an attractive fixed point, the second inequality is just to ensure the vector $V_{\frac{\beta}{2}-\alpha}$ points in the upper right direction to be in coherence with the figure.

We chose $\varepsilon>0$ small (to be made precise later) and $f : \T^1 \to \R$ a continuous function, smooth on $(0,1)$, such that $f = 0$ on $[0,0.5+\varepsilon ]$, $f(t) = (\beta-\alpha)t$ on $[-\varepsilon , 0]$ and 
$\int_0^1 f =0$. The graph of $f$, graph$(f)$, is the red curve.

We then consider $H_2(x,p) = \frac12 p^2 +\frac \alpha2 p(x-x_0)$ where $x_0 = \frac12$.  The conformally symplectic (linear) vector field of $H_2$ is $$X_2(x,p) = 
\begin{pmatrix}
 p+\frac 12 \alpha (x-x_0) \\ -\frac32 \alpha p
 \end{pmatrix}.
 $$ 
 It has $(\frac12,0)$ as hyperbolic fixed point. The unstable direction is $\begin{pmatrix}1 \\0\end{pmatrix}$ associated to the eigenvalue $-\frac{3\alpha}{2}$.
 
 Finally let us consider a continuous function $v : \T^1 \to \R$ that is smooth on $(0,1)$ and such that 
 \begin{itemize}
 \item $v(x) = -\frac12 \beta x$ if $x\in [0,\varepsilon]$,
 \item $v(x) = (\frac12 \beta -\alpha)(\beta-\alpha)x $ for $x\in [-\varepsilon,0]$,
 \item $v(x) = \frac 12 \alpha (x-x_0)$ for $x \in [x_0-\varepsilon,x_0+\varepsilon]$.
 \end{itemize}
The first 2 conditions imply that $v(x) \begin{pmatrix}1 \\ f'(x) \end{pmatrix} = X_1\big(x , f(x)\big)$ for $x\in [-\varepsilon, \varepsilon]$ and the last point asserts similarly that $v(x) \begin{pmatrix}1 \\ f'(x) \end{pmatrix} = X_2\big(x , f(x)\big)$ for $x\in [x_0-\varepsilon,x_0+ \varepsilon]$.

Finally, let us define our final Hamiltonian by 
$$\forall (x,p)\in \T^1, \quad H(x,p)=  \frac12 p^2 +\big(v(x)-f(x)\big)p -v(x)f(x)+\frac12 f(x)^2 -\alpha\int_0^x f(t)dt.$$
 It is quadratic in $p$ with the form sought for :  $H(x,p)= \frac 12 p^2 + a(x)p+b(x)$. Its conformally symplectic vector field is 
$$X(x,p) = \begin{pmatrix}
p+v(x)-f(x) \\
-pv'(x) + (vf)'(x) +\alpha f(x)+\big(p-f(x)\big)f'(x) - \alpha p
\end{pmatrix}.$$
In particular, on the graph of $f$, plugging in the equality $p=f(x)$ we find that 
$$X(x,f(x)) = \begin{pmatrix}
v(x) \\
v(x)f'(x)
\end{pmatrix} = v(x)\begin{pmatrix}1 \\ f'(x) \end{pmatrix}.$$
It follows that the graph of $f$ is invariant by he flow of $X$ and it is ran through with horizontal velocity $v$.

It is obvious from the formulas that $H$ is smooth except maybe on the vertical line $\{x=0\}$. However, note that by construction $X\big(x,f(x)\big) = X_1\big(x,f(x)\big)$ for $x\in [-\varepsilon, \varepsilon]$. Both Hamiltonians are of the same form $H(x,p)= \frac 12 p^2 + a(x)p+b(x)$ and $H_1(x,p)= \frac 12 p^2 + a_1(x)p+b_1(x)$. Writing the conformal Hamiltonian vector fields at points $\big(x,f(x)\big)$ we find that  for $x\in [-\varepsilon, \varepsilon]$
$$\begin{cases}
f(x)+a(x) = f(x) + a_1(x) \\ 
-a'(x)f(x) - b'(x) - \alpha f(x) = -a_1'(x)f(x) + b_1'(x) -\alpha f (x)
\end{cases}
$$
from which it follows that there is a constant $C\in \R$ such that $H(x,p) = H_1(x,p)+C$ for all $(x,p)\in  [-\varepsilon, \varepsilon]\times \R$. In particular $H$ is smooth on the whole annulus.

Finally, we prove that for a good choice of $v$, the graph of $f$ is the global attractor. The proof follows the same lines as the previous example. We first explain how to choose $v$ so that $H$ only has two fixed points $(0,0)$ and $(x_0,0)$. The equations for $X(x,p)=(0,0)$ are 
$$\begin{cases}
0 = p+v(x)-f(x) \\
0 = (f(x)-p)v'(x)+(v(x)-f(x)+p)f'(x)+\alpha f(x) - \alpha p
\end{cases}
$$
that is solved as follows 
$$\begin{cases}
p= -v(x)+f(x) \\
0 = v(x)(v'(x)+\alpha).
\end{cases}
$$
If now we chose $\varepsilon>0$ small enough and $v$ such that $|v'|<\alpha$ we find that the only fixed points of $H$ are indeed $(0,0) $ and $(0,x_0)$. We now argue as in the previous example. By the Poincaré-Bendixon Theorem, the $\omega$-limit set of any orbit has to contain a fixed point of $H$ hence intersects  graph$(f)$. Moreover, it is easily checked that graph$(f)$ is a local attractor, hence all orbits of $H$ are positively attracted to graph$(f)$.

%\begin{figure}[hbt]
%\centering
%\includegraphics[width=0.8\textwidth]{example.pdf}
%\end{figure}

%Fix $q_1\in \mathbb{T}^n$, one can define the action:
%\[
%A_t(q_1, q_2) := \min_{\substack{\gamma: [-t, 0]\rightarrow \mathbb{T}^n \\ \gamma(-t) = q_1 ,\gamma(0) = q_2} } \int_{-t}^0 e^{\alpha s} L(\gamma(s), \dot{\gamma}(s)) \ ds
%\]
%where $\gamma:  [-t, 0]\rightarrow \mathbb{T}^n$ is a continuous piecewise $C^1$ curve and $L$ is the associated Lagrangian of $H$ via the Legendre transformation. 
%
%One could have 
%\[
%\varphi_t (\partial_v L(\gamma(-t), \dot{\gamma}(-t)))  = \partial_v L (\gamma(0), \dot{\gamma}(0)).
%\]
%
%$\partial_p H(q(-t), p(-t))$ is a subdifferential of $-A_t (\cdot, q_2)$ at $q_1$, $\partial_p H(q_0, p_0)$ is a superdifferential of $A_t (q_1, \cdot)$ at $q_2$.
%
%
%We denote by $\widetilde{A}_t(\tilde{q}_1, \tilde{q}_2)$  the associated action on the universal cover. 
%
%Let $u_\alpha$ be the weak KAM solution.
%
%?? One could have the following uniform convergence on compact sets:
%\begin{equation}
%\begin{split}
%\widetilde{A}_t(q_1, q) &\longrightarrow u_\alpha (q) \qquad \text{ as } t\rightarrow +\infty\\
%d\widetilde{A}_t(q_1, q_2) &\longrightarrow d u_\alpha (q_2) \qquad \text{ as } t\rightarrow +\infty
%\end{split}
%\end{equation}

\bibliography{_biblio}
\bibliographystyle{amsplain}
\end{document}